\documentclass[12pt,fleqn,reqno]{article}

\usepackage{epsfig,amssymb,latexsym,amsmath,pifont,mathtools,verbatim,float}

\usepackage[utf8]{inputenc}
\usepackage[T1]{fontenc}

\usepackage[left=2.5cm, right=2.5cm, top=2.5cm, bottom=2.5cm]{geometry}

\usepackage[cal=euler,bb=fourier,scr=zapfc]{mathalfa}

\usepackage[font=small,labelfont=bf]{caption}
\usepackage{graphicx}
\graphicspath{{Figures/}} 
\usepackage{float}

\usepackage{acronym}
\usepackage{paralist}
\usepackage{wasysym}
\usepackage{xspace}
\usepackage{framed}
\usepackage{palatino,pxfonts}
\usepackage{authblk}
\usepackage{enumitem}
\usepackage[dvipsnames,svgnames]{xcolor}
\usepackage{parskip}

\usepackage[numbers,sort&compress]{natbib}

\usepackage{hyperref}
\hypersetup{
colorlinks=true,
linktocpage=true,
pdfstartview=FitH,
breaklinks=true,
pdfpagemode=UseNone,
pageanchor=true,
pdfpagemode=UseOutlines,
plainpages=false,
bookmarksnumbered,
bookmarksopen=false,
bookmarksopenlevel=1,
hypertexnames=true,
pdfhighlight=/O,
urlcolor=DodgerBlue!60!black,linkcolor=DodgerBlue!70!black,citecolor=DarkGreen!70!black
}

\numberwithin{equation}{section}
\usepackage[sort&compress,capitalize,nameinlink]{cleveref}
\crefname{example}{Ex.}{Exs.}

\crefrangeformat{equation}{\upshape(#3#1#4)\textendash(#5#2#6)}

\usepackage[para,online,flushleft]{threeparttable}
\usepackage{multirow}

\usepackage{physics}
\usepackage{float}
\newcommand{\eps}{\varepsilon}
\DeclareMathOperator*{\argmin}{argmin}

\DeclareMathOperator{\zer}{Zer}

\DeclareMathOperator{\diam}{diam}
\DeclareMathOperator{\dist}{dist}

\DeclareMathOperator{\dom}{dom}

\DeclareMathOperator{\gr}{Gr}

\DeclareMathOperator{\resolvent}{\mathsf{J}}

\DeclareMathOperator{\Id}{Id}

\newcommand{\ce}{\mathtt{e}}

\renewcommand{\emptyset}{\varnothing}
\newcommand{\eqdef}{\triangleq}

\newcommand{\scrH}{\mathcal{H}}

\newcommand{\scrK}{\mathcal{K}}

\newcommand{\scrS}{\mathcal{S}}

\newcommand{\R}{\mathbb{R}}

\newcommand{\N}{\mathbb{N}}

\newcommand{\K}{\mathbb{K}}

\DeclareMathOperator{\NC}{\mathsf{NC}}

\newcommand{\opA}{\mathsf{A}}
\newcommand{\opB}{\mathsf{B}}

\newcommand{\opF}{\mathsf{F}}
\newcommand{\opG}{\mathsf{G}}

\newcommand{\opM}{\mathsf{M}}

\newcommand{\opT}{\mathsf{T}}
\newcommand{\opZ}{\mathsf{Z}}

\DeclareMathOperator{\VI}{VI}

\newcommand{\Lip}{L}								

\newcommand{\GapOpt}{\operatorname{Gap}_{\operatorname{opt}}}
\newcommand{\GapFeas}{\operatorname{Gap}_{\operatorname{feas}}}
\newcommand{\Zer}{\operatorname{Zer}}
\newcommand{\Fix}{\operatorname{Fix}}

\usepackage{algorithm}
\usepackage{algpseudocode}

\usepackage{amsthm}
\newtheoremstyle{MyPlainTheorem}{12pt}{6pt}{\itshape}{0pt}{\bfseries}{.}
                { }{\thmname{#1}\thmnumber{ #2}\thmnote{ (#3)}}
\theoremstyle{MyPlainTheorem}
\newtheorem{theorem}{Theorem}
\newtheorem{corollary}[theorem]{Corollary}
\newtheorem*{corollary*}{Corollary}
\newtheorem{lemma}[theorem]{Lemma}
\newtheorem{proposition}[theorem]{Proposition}

\newtheorem{definition}[theorem]{Definition}
\newtheorem*{definition*}{Definition}
\newtheorem*{problem*}{Problem}
\newtheorem{assumption}{Assumption}
\newtheorem{remark}{Remark}
\newtheorem*{remark*}{Remark}
\newtheorem*{notation*}{Notational remark}
\newtheorem{example}{Example}

\numberwithin{theorem}{section}
\numberwithin{remark}{section}
\numberwithin{example}{section}

\DeclarePairedDelimiter{\inner}{\langle}{\rangle}

\renewcommand{\norm}[1]{\|{#1}\|}
\renewcommand{\bar}{\overline}

\title{Regularization methods for solving hierarchical variational inequalities with complexity guarantees}
\date{\today}

\author[1]{\small Daniel Cortild}
\author[2]{\small Meggie Marschner,  Mathias Staudigl}

\affil[1]{\footnotesize University of Oxford, Mathematical Institute, Woodstock Road, OX2 6GG Oxford, UK\\
(\href{mailto:daniel.cortild@maths.ox.ac.uk}{daniel.cortild@maths.ox.ac.uk})}
\affil[2]{\footnotesize Mannheim University, Department of Mathematics, B6 26, 68159 Mannheim, DE\\
(\href{mailto:m.marschner@uni-mannheim.de}{m.marschner@uni-mannheim.de}, \href{mailto:m.staudigl@uni-mannheim.de}{m.staudigl@uni-mannheim.de})}

\begin{document}

\maketitle

\begin{abstract}
    We consider hierarchical variational inequality problems, or more generally, variational inequalities defined over the set of zeros of a monotone operator. This framework includes convex optimization over equilibrium constraints and equilibrium selection problems. In a real Hilbert space setting, we combine a Tikhonov regularization and a proximal penalization to develop a flexible double-loop method for which we prove asymptotic convergence and provide rate statements in terms of gap functions. Our method is flexible, and effectively accommodates a large class of structured operator splitting formulations for which fixed-point encodings are available. Finally, we validate our findings numerically on various examples.
\end{abstract}


\section{Introduction}\label{sec:intro}

Bi-level optimization problems consist of two nested optimization formulations, referred to as the inner and the outer problem. Solutions to the inner problem determine the feasible set over which a minimizer of an outer objective function is searched for. This class of problems is an extremely active area of research recently, driven by important practical applications in engineering, economics and machine learning. Comprehensive reviews can be found in \cite{dempe2020bilevel,liu2021investigating}. A particularly appealing class of bi-level optimization problems are those in which only a single decision variable is involved at both layers of the problem. These special cases are often called simple bi-level optimization problems \cite{Dempe:2021aa}. In this special instance, the problem is given by the convex constrained minimization problem 
\begin{equation*}
    \min_{x\in \scrS_{0}}g(x), \quad\text{where}~ \mathcal S_0=\argmin_{x\in\scrH}\{f(x)+r(x)\},
\end{equation*}
 $\scrS_{0}$ is the set of minimizers of a the inner level problem, where $f$ and $g$ are convex and continuously differentiable and $r$ is an extended-valued proper closed convex function. This problem received significant attention over the last years, with important advances made in \cite{beck2014first,SabSht17,MerchavSIOPT23,AminiYous19,shen2023online} and \cite{yousefian2021bilevel}. Moving beyond the convex optimization formulation of the inner and the outer level problem, recent papers consider a more general problem formulation by replacing them by variational inequalities. While this is a classical formulation in inverse problems and signal processing \cite{deutsch1998minimizing,marino2006general,Yamada2011MinimizingTM}, only a few papers obtain rates of convergence in terms of merit functions. Motivated by narrowing this gap, we consider the problem of solving a class of constrained variational inequality problems 
\begin{equation}\label{eq:P}\tag{P}
    \VI(\opG,\scrS_{0})\colon \text{ Find $u\in\scrS_{0}$ s.t. }\inner{\opG u,v-u}\qquad\forall v\in\scrS_{0}=\zer(\opM),
\end{equation}
where $\opG\colon\scrH\to\scrH$ is a monotone and Lipschitz continuous operator, and $\scrS_{0}=\zer(\opM)$ is the (non-empty) set of zeros of another maximally monotone operator $\opM\colon\scrH\to 2^{\scrH}$, all defined on a real Hilbert space $\scrH$ with inner product $\inner{\cdot,\cdot}$ and corresponding norm $\norm{\cdot}$. Besides the simple bi-level optimization problem mentioned above, the family of problems \eqref{eq:P} contains many specific instances heavily studied in the literature. 

\begin{example}[Hierarchical Variational Inequality]
    Given a mapping $\opF\colon\scrH\to\scrH$ and a closed convex set $\scrK\subseteq\scrH$,  the variational inequality problem $\VI(\opF,\scrK)$, is the problem of finding a point $u\in\scrK$ such that 
     \[
     \inner{\opF(u),v-u}\geq0\qquad\forall v\in\scrK. 
     \]
     The solution set of $\VI(\opF,\scrK)$ can be expressed as $\zer(\opF+\NC_{\scrK})$, involving the normal cone operator 
    \[
     \NC_{\scrK}(v)\eqdef\left\{\begin{array}{ll} 
     \emptyset & \text{if }v\notin\scrK\\
     \{p\in\scrH\vert \inner{p,u-v}\leq 0,\;\forall u\in\scrK\} & \text{ else.}
     \end{array}\right.
     \]
    Hence, by setting $\opM=\opF+\NC_{\scrK}$, Problem \eqref{eq:P} specializes to a hierarchical variational inequality problem 
    \begin{equation*}
    \VI(\opG,\scrS_{0}),\quad\text{where $\scrS_{0}$ is the solution set of $\VI(\opF,\scrK)$}.
    \end{equation*}
    For recent studies of such problem formulations in finite-dimensional Euclidean spaces, we refer to \cite{alves_inertial_2025} and \cite{samadi_improved_2025}.
\end{example}

\begin{example}[Equilibrium Selection]
    A mathematical program with equilibrium constraints \cite{Luo_Pang_Ralph_1996} optimizes a real-valued function over a feasible set described as the solution set of an equilibrium problem. As such, one can select amongst lower-level equilibria via a convex design criterion $g\colon\mathcal{H}\to\mathbb{R}$ by writing
    \[
    \min_{u\in\scrH}\left\{g(u) ~~\colon~~ u \in\scrS_{0} \right\}.
    \]
    If $g$ is convex and Fr\'echet differentiable, this becomes a hierarchical variational inequality with $\opG=\nabla g$, hence a particular instance of Problem \eqref{eq:P}. This viewpoint covers a broad range of models encountered in, e.g., engineering, economics, inverse problems, and control (see \cite{hintermuller2011first,hintermuller2015bilevel,hintermuller2014several,outrata2000generalized,de2023bilevel}).
\end{example}

\subsection{Contributions}

Numerical schemes with complexity statements for resolving Problem \eqref{eq:P} have been recently published in the papers \cite{Kaushik:2021aa,samadi_improved_2025} and \cite{alves_inertial_2025}. \cite{Kaushik:2021aa} extend the simple bi-level setting to a lower level problem defined in terms of a variational inequality, while keeping the upper level problem as a convex optimization problem. This is a particular case of Problem \eqref{eq:P}, which is related to mathematical programming with equilibrium constraints \cite{Luo_Pang_Ralph_1996}, and to the equilibrium selection problem in game theory \cite{Benenati:2022aa,Benenati:2023aa}. The recent papers \cite{samadi_improved_2025,alves_inertial_2025} go one step further and study a special instance of Problem \eqref{eq:P} in which the upper- and lower-level problems are defined in terms of a variational inequality.

In this paper we generalize the aforementioned existing results, by considering a significantly larger class of hierarchical equilibrium problems, admitting the formulation \eqref{eq:P}. We prove rates on both levels of the equilibrium problem via a very flexible numerical scheme, which can easily be modulated to prevailing problem structure (specifically of the lower-level inclusion problem). Our main contributions are threefold:
\begin{itemize}
    \item \textbf{Hilbert space setting.} We extend analyses that are typically confined to finite-dimensional Euclidean spaces to possibly infinite-dimensional real Hilbert spaces.
    \item \textbf{Monotone-inclusion lower level and modular oracle.} Rather than considering a variational inequality at the lower level, we allow the feasible set to be defined as the solution set of an abstract monotone inclusion problem. Within this general framework, we adopt a very general algorithmic design based on a monotone-inclusion oracle obtained through a fixed-point encoding, which can be tailored to the structure of the operators involved. As a result, our numerical scheme and our analysis is applicable to entire families of relevant splitting methods, and not just a single one. 
    \item \textbf{Quantitative guarantees via gap functions.} While much of the existing literature focuses on weak convergence results without explicit rate statements, we introduce gap functions for both the lower- and upper-level problems and establish quantitative convergence guarantees for these measures. Our results are formulated within a specific geometric setting in which we assume that the gap function associated with the lower-level problem satisfies a Hölderian error bound condition. Such settings have already been considered in \cite{Cabot2005} in the context of hierarchical optimization, and more recently in \cite{samadi_improved_2025} and \cite{boct2025accelerating}.
\end{itemize}

We achieve these results via the careful design of an algorithmic template drawing inspiration from classical diagonal methods, involving iterative Tikhonov and proximal regularization via anchor terms. By constructing a family of auxiliary combined equilibrium problems, our method constructs a sequence in an inner loop method, which is designed to track a sequence of temporal solutions. This procedure is essentially an inertial Krasnoselskii-Mann (KM) iteration and thus only relies on the availability of a ready-to-compute fixed-point encoding map for the lower-level solution set $\scrS_{0}=\zer(\opM)$.   An outer-loop restarting procedure updates the Tikhonov parameter and the anchor term to steer the trajectory of the numerical algorithm towards a solution of \eqref{eq:P}.
 
Diagonal schemes for solving convex programs are rather classical (see e.g. \cite{alart1991penalization}). The idea of extending this approach to hierarchical equilibrium problems is taken from \cite{Facchinei:2014aa} and \cite{Lampariello:2020aa}. In particular, \cite{Lampariello:2020aa} also proposes a kind of double-loop scheme with a forward-backward algorithm acting in the inner loop. In this paper, we consider a more general class of splitting problems, with explicit rate guarantees at both levels of the problem. Specifically, in the case where $\opG$ is merely monotone (as opposed to strongly monotone), we obtain weak convergence of the averaged sequence of iterates $(\overline w_n)_{n\in\N}$, with rates of the form, for arbitrary $b\in (0, 1)$,
\[
    \GapOpt(\overline{w}_{n})\leq \tilde{\mathcal O}\left(\frac1{n^{1-b}}\right), \quad \text{and}\quad 0\leq \GapFeas(\overline{w}_{n})\leq \tilde{\mathcal O}\left(\frac1{n^{\min(b, 1-b)}}\right),
\]
for suitably defined optimality and feasibility gaps (see Section \ref{S:conv} for precise definition). In this setting, the integer $n$ measures the number of restarts, and $\tilde {\mathcal O}$ disregards logarithmic factors, which are consequences of the double-loop structure of our algorithm. Besides these logarithmic factors, our work is on par with published single-loop algorithms:
\begin{itemize}
    \item In \cite[Corollary 3.5]{Kaushik:2021aa}, the authors study optimization problems with Cartesian variational inequality constraints, where the objective is convex and the constraint set is a Cartesian product associated with a monotone mapping. Their algorithm resembles a forward-backward method, cast into a single-loop scheme. For the deterministic variant of their method, they obtain rates of $\mathcal O(1/N^{0.5-b})$ for the optimality gap and $\mathcal O(1/N^{b})$ for the feasibility gap, for $b\in(0,0.5)$, $N$ being the iteration count. Therefore, ignoring logarithmic factors, the rates achieved in the our work are more favorable.
    \item In \cite[Theorem 3.6]{samadi_improved_2025}, the authors study a regularized extragradient method, and establish rates of $\mathcal O(1/N^{1-b})$ for the optimality gap and $\mathcal O(1/N^{b})$ for the feasibility gap. In the $b\in (0, 0.5]$ regime our rates match with theirs, and we improve upon their rates in the $b\in (0.5, 1)$ regime. We do again recall that our rates include logarithmic terms that their avoids. We note that \cite[Theorem 3.8]{alves_inertial_2025} studies an inertial version of the same method, but does not improve upon the rates.
    \item The recent work \cite[Theorem 9]{marschner_tikhonov_2025} analyzes a special case of our general algorithmic template, focusing on the forward-backward splitting only. The analysis provided is weaker, in particular, no rate statements are contained in that work.
\end{itemize}
When $\opG$ is strongly monotone, our rates improve to 
\[
    \GapOpt(\overline{w}_{N})\leq \tilde{\mathcal O}\left(\frac1{N}\right), \quad \text{and} \quad 0\leq \GapFeas(\overline{w}_{N})\leq \tilde{\mathcal O}\left(\frac {1}N\right),
\]
where, again, $N$ measures the number of restarts. These rates are comparable to those in \cite[Theorem 4.6]{samadi_improved_2025}, where a regularized extragradient method is analyzed. 

\paragraph{Organization of the paper.}
The remainder of the paper is organized as follows. The rest of this section includes notation and definitions. A general presentation of the double loop scheme is presented in Section \ref{sec:algo}, and its analysis is given in Section \ref{S:analysis}. Section \ref{sec:splitting} gives examples for the fixed-point encoding map used in the inner loop of our method. Numerical experiments are then performed in Section \ref{sec:numerics}. We conclude in Section \ref{sec:conc}.

\subsection{Notation and Definitions}
We adopt standard notation and terminology from variational analysis, as in \cite{BauCom16}. Let $\scrH$ be a real Hilbert space with inner product $\inner{\cdot,\cdot}$ and corresponding norm $\norm{\cdot}$. A set-valued operator $\opA\colon\scrH\to 2^{\scrH}$ is $\mu$-monotone if $\mu\geq 0$ and
\[
\inner{b-a,u-v}\geq \mu\norm{u-v}^{2} \text{ for all } (u, b), (v, a)\in \gr(\opA)\eqdef\{(v,a)\in\scrH\times\scrH\vert a\in\opA(v)\}.
\]
If $\mu=0$, we will say $\opA$ is merely monotone. The operator $\opA\colon\scrH\to 2^{\scrH}$ is maximally monotone if it is monotone and there exists no other monotone operator whose graph properly contains $\gr(\opA)$. The domain of an operator $\opA$ is $\dom(\opA)=\{x\in\mathcal{H}\vert \opA(x)\neq\emptyset\}$. A mapping $\opT\colon\scrH\to\scrH$ is \textit{quasi-contractive} if $\|Tx-p\|\leq q\|x-p\|$ for all $p\in\textnormal{Fix}(T)$ and some $q\in(0,1)$, where $\textnormal{Fix}(\opT)\eqdef\{x\in\scrH\vert x=\opT(x)\}$ is the set of fixed points of $\opT$. If the same inequality holds for $q=1$, then $\opT$ is \emph{non-expansive}.\\
The resolvent of an operator $\opA$ is defined as $\resolvent_{\opA}\eqdef(\Id+\opA)^{-1}$. By Minty's theorem, if $\opA$ is maximally monotone, then the resolvent is a nonexpansive operator and thus single-valued. The normal cone mapping to a closed convex and non-empty set $\mathcal{K}\subset\mathcal{H}$ is defined as 
$\NC_{\mathcal{K}}(x)=\{u\in\mathcal{H}\vert \inner{u,y-x}\leq 0 \;\forall y\in\mathcal{K}\}$ for $x\in\mathcal{K}$ and $\NC_{\mathcal{K}}(x)=\emptyset$ otherwise. Under convexity, $\NC_{\mathcal{K}}$ is maximally monotone, and its resolvent coincides with the orthogonal projector $\Pi_{\mathcal{K}}(x)=\argmin\{\norm{y-x}:y\in\mathcal{K}\}$. 

For two sequences $(a_n)\subset \R$ and $(b_n)\subset \R$, we say that $a_n={\mathcal O}(b_n)$ if there exists a constant $C>0$ such that $a_n\le Cb_n$ for $n$ sufficiently large. We say that $a_n=o(b_n)$ if $\lim_{n\to \infty}{b_n}/{a_n}=0$. Finally, we say that $a_n=\Theta(b_n)$ if there exists a constant $C>0$ such that $\lim_{n\to \infty}{b_n}/{a_n}=C$.

\section{Outline of the Algorithm}\label{sec:algo}

In this section, we first formulate the main assumptions imposed on Problem \eqref{eq:P}. 

\begin{assumption}\label{ass:1}
We assume the following on Problem \eqref{eq:P}:
\begin{enumerate}
 \item $\opG:\scrH\to\scrH$ is a $\mu$-strongly monotone, with $\mu\geq 0$, and $L_{\opG}$-Lipschitz continuous operator;
 \item $\opM\colon\scrH\to2^{\scrH}$ is maximally monotone with bounded domain;  
 \item The set $\scrS_{0}$ is nonempty. 
 \end{enumerate}
 \end{assumption}
 \begin{remark}
 Assumption \ref{ass:1} guarantees that the overall Problem \eqref{eq:P} has a solution. In particular, $\scrS_{0}$ is a  convex and compact subset of $\scrH$ \cite{BauCom16}. 
 \end{remark} 

For given parameters $(\alpha,\beta)\in(0,\infty)^{2}$ and an anchor point $w\in\scrH$, we define the associated auxiliary problem as finding
\begin{equation}\label{eq:Aux-P}\tag{Aux}
\bar u_{\beta}(w)\in \Zer(\opM+\beta\opG+\alpha(\Id-w)).
\end{equation}
Here, $\alpha>0$ is a proximal parameter enforcing strong monotonicity, ensuring that the point $\bar u_{\beta}(w)$ is the unique zero of $\opM+\beta\opG+\alpha(\Id-w)$. The parameter $\beta>0$ is a Tikhonov parameter that balances the importance of the lower- and upper-level problems. 

Our numerical scheme consists of two main procedures. The inner procedure is an inertial KM iteration, whose last iterate is designed to lie in a sufficiently small ball around the temporal solution $\bar u_{\beta}(w)$, with iteratively defined radius. We subsequently use this last iterate as a new anchor point, and restart the KM iteration with fresh parameters. To this end, we assume to have access to a family of parameterized fixed-point mappings $\opT_{k}\equiv\opT^{(w,\beta)}_{k}\colon\scrH\to\scrH$, satisfying the following condition.

\begin{assumption}\label{ass:contraction}
For all $k\in \N$ and $(w,\beta)\in\scrH\times\R_{++}$, the operator $\opT^{(w,\beta)}_{k}$ is a quasi-contraction with parameter $q_{k}\in(0,\bar q)$ for some $\bar q\in(0,1)$, sharing a common fixed point $p(w,\beta)$ for all $k\in \N$. Moreover, $\opT_{k}^{(w,\beta)}$ is related to Problem \eqref{eq:Aux-P} via a nonexpansive \textit{fixed-point transportation map} $\opZ^{(w,\beta)}\colon \scrH\to\scrH$ such that $\bar u_{\beta}(w)=\opZ^{(w,\beta)}(p(w,\beta))$.
\end{assumption}
Concrete examples of such fixed-point encoding mappings are provided in Section \ref{sec:splitting}. 
\begin{remark}
The proximal regularization parameter $\alpha>0$ is treated in our approach as an external input to the algorithm. Hence, its concrete numerical value used in the computation affects the position of the solution $\bar{u}_{\beta}(w)$, the shape of the fixed-point encoding map $\opT^{(w,\beta)}_{k}$, and eventually also the fixed-point transportation map $\opZ^{(w,\beta)}$. However, since it is a fixed parameter, we simplify the notation and omit an explicit dependence from the involved mappings and operators on the value of $\alpha$. 
\end{remark}

\subsection{Inner Loop}\label{ssec:inner}

The inner loop of our tracking algorithm employs the procedure $\mathtt{KM}(v,(\opT^{(w,\beta)}_{k})_{k},(\tau_{k})_{k},(\theta_{k})_{k},\varepsilon)$, see Algorithm \ref{algo:inner_loop}, which is an inertial Krasnoselskii-Mann iteration with user-provided parameters embodied in terms of the initial point $v\in\scrH$, fixed-point encoding mappings $(\opT^{(w,\beta)}_{k})_{k}$ satisfying Assumption \ref{ass:contraction}, as well as momentum parameters $(\tau_{k})_k$ and relaxation parameters $(\theta_{k})_{k}$. This fixed-point iteration is  designed to return an approximation of the central funnel fixed-point $\bar{u}_{\beta}(w)$ via the fixed-point transportation map $\opZ^{(w,\beta)}$. 

\begin{algorithm}[H]
\caption{Function $\mathtt{KM}(v,(\opT^{(w,\beta)}_{k})_{k},(\tau_{k})_{k},(\theta_{k})_{k},\varepsilon)$}
\label{algo:inner_loop}
\begin{algorithmic}
\State $v_1= v_0\eqdef v$ 
\State Set $k=1$ 
\While{$\|v_{k+1}-z_k\| > \varepsilon$}
 \State $k\leftarrow k+1$
 \State $z_{k}\leftarrow v_{k}+\tau_{k}(v_{k}-v_{k-1})$
 \State $v_{k+1}\leftarrow (1-\theta_{k})z_{k}+\theta_{k}\opT^{(w,\beta)}_{k}(z_{k})$
\EndWhile
\State \textbf{return} $v_{k+1}$
\end{algorithmic}
\end{algorithm}
Algorithm \ref{algo:inner_loop} employs the stopping time
\begin{equation}\label{eq:stopping}
\K(\eps)\eqdef\inf\{k\geq 1\vert\;\norm{v_{k+1}-z_{k}}\leq\eps\}.
\end{equation}
This stopping criterion is motivated by the following approximation result, formulated in terms of general fixed-point encodings obeying the imposed assumptions. 
\begin{lemma}\label{lem:stopping}
Let $(\tau_k)_k\subset [0, 1], (\theta_k)_k\subset [0, \bar\theta]$ for $\bar\theta\in (0, 1)$, $\eps>0$, and let $(\opT_k)_k$ be a sequence of operators satisfying Assumption \ref{ass:contraction}. Denote by $p$ the common fixed-point of $(\opT_k)_k$. If the stopping time $\K(\eps)$ defined in \eqref{eq:stopping} is finite and $v_{\K(\eps)+1}$ is the last iterate produced by Algorithm \ref{algo:inner_loop}, then
\begin{equation}\label{eq:compute_en}
\norm{v_{\K(\eps)+1}-p}\leq\ce\eqdef \frac{\eps}{\bar\theta}\left({1-\bar\theta}+\frac{\bar{q}}{1-\bar{q}}\right).
\end{equation}
\end{lemma}
\begin{proof}
Given $\eps>0$, let $\K=\K(\eps)$. By definition of $v_{k+1}$ and quasi-contractiveness of $\opT_k$, we get
\begin{align*}
\norm{z_k-p}&=\norm{z_k-\opT_{k}(z_k)+\opT_{k}(z_k)-p}\leq\norm{\frac{v_{k+1}-z_k}{\theta_k}}+\norm{\opT_{k}(z_k)-p}\\
&\leq\frac{\eps}{\theta_k}+\bar{q}\cdot  \norm{z_k-p}.
\end{align*}
Whence $\norm{z_k-p}\leq\frac{\eps}{\theta_k(1-\bar{q})}$. We conclude that
\begin{align*}
\norm{v_{\K+1}-p}&=\norm{v_{\K+1}-\opT_{k}(z_\K)+\opT_{k}(z_\K)-p}\\
&\leq\norm{\frac{(1-\theta_\K)(z_\K-v_{\K+1})}{\theta_\K}}+\norm{\opT_{k}(z_\K)-p}\\
&\leq\eps\frac{1-\theta_{\K}}{\theta_\K}+\bar{q}\cdot \norm{z_\K-p}\\
&\leq\frac{\eps}{\theta_\K}\left({1-\theta_{\K}}+\frac{\bar{q}}{1-\bar{q}}\right)\leq \frac{\eps}{\bar\theta}\left({1-\bar\theta}+\frac{\bar{q}}{1-\bar{q}}\right).
\end{align*}
\end{proof}
Lemma \ref{lem:stopping} has important consequences for the tracking properties of the anchor points $w$ which we recursively generate with the restarting procedure explained next. To give an outlook, consider the sequence $(v_{k})_{k}$ produced by the function $\mathtt{KM}(v,(\opT^{(w,\beta)}_{k})_{k},(\tau_{k})_{k},(\theta_{k})_{k},\varepsilon)$, and let $w^{+}\eqdef \opZ^{(w,\beta)}(v_{\K(\eps)+1})$. Lemma \ref{lem:stopping} guarantees that this point satisfies the tracking guarantee
\begin{equation}\label{eq:temporalerror}
\norm{w^{+}-\bar{u}_{\beta}(w)}\leq\ce=\frac{\eps}{\bar\theta}\left({1-\bar\theta}+\frac{\bar{q}}{1-\bar{q}}\right),
\end{equation}
by nonexpansiveness of $\opZ^{(w, \beta)}$. The upper bound is of order $\mathcal{O}(\eps)$, and thus controlled by the user-defined targeted precision. 

\subsection{Restarting Procedure}

Given an anchor point $w$, Algorithm \ref{algo:inner_loop} produces a sequence of points $(v_{k})_{k}$, whose last iterate corresponds to the output of the function $\mathtt{KM}(v,(\opT^{(w,\beta)}_{k})_{k},(\tau_{k})_{k},(\theta_{k})_{k},\varepsilon)$. In view of the tracking property established in \eqref{eq:temporalerror}, we will use this point as the new anchor point for the inner-loop KM iteration. Repeating this over time for a fixed number of restarts with updated parameters required by the inner-loop, defines our outer loop procedure. The mechanics of this procedure gives rise to a diagonal equilibrium tracking method (\texttt{DANTE}), formally described in Algorithm \ref{algo:outer_loop}.

\begin{algorithm}[H]
\caption{\texttt{DANTE} (DiAgoNal equilibrium Tracking mEthod)}
\label{algo:outer_loop}
\begin{algorithmic}
\State $w_{0}\in\scrH$ given initial point. $(\tau_{0,k}),(\theta_{0,k})_{k}$ given parameters and $N\geq 1$ given number of restarts.
\State Set $\lambda_{0}=1$ and $S_{0}=0$.
\State Set $\bar{w}_{0}=w_{0}$.
\For {$n=0, \ldots, N-1$}
   \State Set $v=\mathtt{KM}(w_{n},(\opT^{(w_{n},\beta_{n})}_{n,k})_{k},(\tau_{n,k})_{k},(\theta_{n,k})_{k},\eps_{n})$.
   \State Let $w_{n+1}=\opZ^{(w_{n},\beta_{n})}(v)$ denote the next anchor point.
   \State Update $\lambda_{n+1}=\lambda_{n}(1+\frac{2\mu\beta_{n}}{\alpha})$ and $S_{n+1}=S_{n}+\lambda_{n}\beta_{n}$
   \State Update $$
   \bar{w}_{n+1}=\frac{S_{n}\bar{w}_{n}+\lambda_{n}\beta_{n}w_{n+1}}{S_{n+1}}.
   $$
    \State Compute $\ce_n$ based on Equation \eqref{eq:compute_en}.
    \State Update $\beta_{n+1},(\tau_{n+1,k})_{k},(\theta_{n+1,k})_{k},\eps_{n+1}$. 
\EndFor
\State \textbf{return} {$\bar{w}_{N}$}
\end{algorithmic}
\end{algorithm}

Note that \texttt{DANTE} is a method rather than a bona-fide algorithm as we do not specify how the parameters required by the fixed-point iteration $\mathtt{KM}(v,(\opT^{(w,\beta)}_{k})_{k},(\tau_{k})_{k},(\theta_{k})_{k},\varepsilon)$ are updated. This is a part of the user-defined input and should be decided based on the eventually known problem structure. 

\section{Convergence Analysis}\label{S:analysis}

In this section, we focus on the convergence analysis of our proposed method. The double-loop structure naturally splits the analysis into the analysis of the inner loop (Section \ref{S:inner}), the analysis of the restarting procedure (Section \ref{S:conv}), and the analysis of how they interact (Section \ref{ssec:comb}).

\subsection{Convergence Analysis of Inner Loop}\label{S:inner}

Before proving that Algorithm \ref{algo:inner_loop} terminates in finite time, we recall the following assumption from \cite{Cortild:2024aa}. Note that the conditions for Lemma \ref{lem:stopping} are verified under Assumptions \ref{ass:contraction} and \ref{ass:KM}, when considering a single outer iteration. 

\begin{assumption}{\cite[Hypothesis 2.1]{Cortild:2024aa}}\label{ass:KM}
    Assume $(\theta_k)_k\subset(0,\bar\theta]$ and that $(\tau_k)_k\subset [0,\bar\tau]$ is monotonically non-decreasing, where $\bar\theta, \bar\tau\in (0, 1)$. Moreover, 
    \begin{equation}\label{eq:parameterineq}
    Q_k\tau_k(1+\tau_k)+(\theta_k^{-1}-1) \tau_k(1-\tau_k)-Q_k(\theta_{k-1}^{-1}-1)(1-\tau_{k-1})< 0,
    \end{equation}
   for all $k\in \N$, where $Q_k=1-\theta_k(1-q_k^2)$, and that $\bar Q\eqdef \sup_{k\in \N}Q_k<1$. 
\end{assumption}

Under the above assumption, we can guarantee strong convergence of the iterates at a linear rate.

\begin{theorem}{\cite[Theorem 2.2]{Cortild:2024aa}}\label{th:MainInexact}\label{thm:822}
    Assume Assumptions \ref{ass:contraction} and \ref{ass:KM} hold. 
    If $(z_k,v_{k})_{k}$ is generated by Algorithm \ref{algo:inner_loop} without stopping criterion, then $(z_k,v_k)_k$ converges strongly to $(p, p)=(\textnormal{Fix}(\opT_{k}),\textnormal{Fix}(\opT_{k}))$. Moreover 
    $\sum_{k\geq 1}\norm{v_k-p}^{2}<\infty$, and specifically
    \[
        \norm{v_{k}-p}^2\leq\frac{\bar{Q}^{k}}{(1-\bar\tau)(1-\bar\theta)}\cdot \norm{v_{0}-p}^2
    \]
for all $k\ge 1$.
\end{theorem}

An immediate corollary is that the stopping criterion in Algorithm \ref{algo:inner_loop} allows to conclude finite-time termination.
\begin{corollary}\label{coro:conv}
    Assume Assumptions \ref{ass:contraction} and \ref{ass:KM} hold. It holds that $\sum_{k\ge 1}\|v_{k+1}-z_k\|^2<\infty$.
    Specifically, $\|v_{k+1}-z_k\|\to 0$, and hence Algorithm \ref{algo:inner_loop} terminates in finite time. Moreover, it holds that
    \[
        \norm{v_{k+1}-z_k}^2\le \bar Q^{k-1}\cdot \frac{4(1+\bar\tau)^2}{(1-\bar\tau)(1-\bar\theta)}\cdot \norm{v_0-p}^2
    \]
    for all $k\ge 2$.
\end{corollary}
\begin{proof}
    Note that 
    \begin{align*}
        \|v_{k+1}-z_k\|
        &\le \|v_{k+1}-p\|+\|z_k-p\| \\
        &\le \|v_{k+1}-p\|+(1+\tau_k)\|v_k-p\|+\tau_k\|v_{k-1}-p\| \\
        &\le 2(1+\bar \tau)\max(\|v_{k+1}-p\|, \|v_{k}-p\|, \|v^{k-1}-p\|).
    \end{align*}
    The result follows by the square-summability of $(v_k-p)_{k}$, and the rate provided in Theorem \ref{th:MainInexact}.
\end{proof}

The rate in Corollary \ref{coro:conv} does not provide a user-friendly quantitative result unless one can control the quantity $\|v_0-p\|$. We present an additional set of assumptions under which this can be done.

\begin{assumption}\label{ass:add}
    We assume that $\dom(\opT^{(w,\beta)}_k)\subset \dom(\opM)$ for all $k\in \N$ and all $(w,\beta)$. 
\end{assumption}

\begin{remark}
    Assumption \ref{ass:add} might seem restrictive, but in fact, they are satisfied by many operators. As a concrete example, assume that we know that the lower level is formulated in terms of the operator $\opM=\opF+\NC_{\mathcal K}$, where $\NC_{\mathcal K}$ is the normal cone of the closed convex set ${\mathcal K}=\dom(\opM)$ and $\opF:\mathcal{H}\to\mathcal{H}$ is a Lipschitz continuous monotone operator. With this problem structure, a powerful fixed-point encoding strategy is the forward-backward operator. Adapted to the auxiliary problem \eqref{eq:Aux-P}, this method is defined in terms of the mapping 
    \[
    \opT^{(w,\beta)}_{k}(x)=\Pi_{\mathcal{K}}(x-\gamma_{k}(\opF(x)+\beta\opG(x)+\alpha(x-w))), 
    \]
    $\gamma_{k}>0$ being a suitably chosen step-size parameter. Clearly, $\opT^{(w,\beta)}_{k}(x)\in {\mathcal K}$ for all $x\in \scrH$. The forward-backward operator and other examples are discussed in detail in Section \ref{sec:splitting}.
\end{remark}

Under these assumptions, we can achieve a user-friendly rate, as presented in the following corollary.

\begin{corollary}\label{coro:complexity}
    Let Assumptions \ref{ass:contraction}-\ref{ass:add} be in place. Then, 
    \[
        \|v_{k+1}-z_k\|\le C\cdot \bar Q^{k/2},
    \]
    for all $k\ge 1$, where 
    \begin{equation}\label{eq:C}
        C\eqdef \frac{2(1+\bar\tau)}{\sqrt{\bar Q(1-\bar\tau)(1-\bar\theta)}}\cdot \diam(\dom(\opM)).
    \end{equation}
    Specifically, 
    \begin{equation}\label{eq:Kn}
    \K(\eps)\leq \bar{K}_{\eps}\eqdef  \left\lceil2\log(C/\varepsilon)/\log(1/\bar Q)\right\rceil.
    \end{equation}
\end{corollary}
\begin{proof}
    The conditions guarantee that all iterates are in $\dom(\opM)$, thus justifying the bounds. Moreover, since $\bar{Q}\in(0,1)$, $C\cdot \bar Q^{k/2}\leq\varepsilon$ implies the bound on $k$.
\end{proof}

\begin{remark}\label{rem:nonexpansive}
    By \cite[Theorem 2.1]{Cortild:2024aa}, we also know that weak convergence of the iterates $(v_k)$ and the velocities $(z_k)$ is guaranteed when $\opT_{k}$ is just nonexpansive. Moreover, if $\opT_k$ is a fixed-point encoding of Problem \eqref{eq:Aux-P}, which has a unique solution $p$ by construction, we know that the sequence convergences weakly to this unique fixed-point. Specifically, in the finite-dimensional setting, we know that $\|v_k-p\|\to 0$ and $\|v_{k+1}-z_k\|\to 0$, meaning that Algorithm \ref{algo:outer_loop} still can be applied. However, Lemma \ref{lem:stopping} no longer holds, and we do not have an appropriate stopping criterion to guarantee $\|w^+-\bar u_\beta(w)\|\le \ce$. Within our final numerical implementation in Section \ref{sec:image}, we experimentally test our algorithm on problems displaying only nonexpansive operators, corresponding to $q=1$, noting that the theoretical guarantees are weakened in this setting.
\end{remark}

\begin{remark}
    Rather than assuming exact access to the operators $\opT_k$, we could assume access to $\delta_k$-perturbations $\tilde\opT_k$, with bounded inexactness model $0\leq\delta_{k}\leq \delta_{\max}$. Such an error structure is motivated by the fact that the exact map might in general not be available, or its computation may be very demanding. This occurs for instance when applying proximal methods to image deblurring with total variation \cite{Chambolle:2004aa}, or to structured sparsity regularization problems in machine learning and inverse problems \cite{Zhao:2009aa}. In those cases, the proximity operator is usually computed using ad hoc algorithms, and therefore inexactly. The implications of inexact implementations are the following: Lemma \ref{lem:stopping} would still hold, by replacing $\varepsilon$ by $\max(\varepsilon, \delta_{\max})$, provided $\delta_{\max}$ is finite. Additionally, Theorem \ref{thm:822} would still hold \citep{Cortild:2024aa}, provided $(\theta_k\delta_k)\in \ell^2$, but without an explicit rate. As such, Corollary \ref{coro:conv} still holds, but again without rate guarantees. Specifically, we have no a-priori upper bound on the required number of iterations required for Algorithm \ref{algo:inner_loop} to terminate, but we know it will terminate.
\end{remark}

\subsubsection{Parameters Analysis}

We want to take a closer look at the numerical parameters that can be chosen in the inner loop. We restrict ourselves to the case of constant parameters, namely $\tau_{k}\equiv\tau\in(0,1)$, $\theta_{k}\equiv\theta\in(0,1)$, and $q_k\equiv q\in (0,1)$, implied by $\opT_k\equiv \opT$. In this case, Inequality \eqref{eq:parameterineq} in Assumption \ref{ass:KM} becomes 
\[
-\tau^2\cdot (1-2\theta+\theta^2(1-q^2))+\tau\cdot (2-(2-q^2)\theta)-(1-(1-q^2)\theta)(1-\theta)<0.
\]
As such, for a fixed $(\theta, q)$, the admissible values of $\tau$ are $\tau\in [0, \bar \tau)$, where $\bar\tau$ is the largest root of the given quadratic. In Figure \ref{fig:PA}, we plot $\bar\tau$ for different values of $q$.
\begin{figure}[H]
    \centering
    \includegraphics[width=0.48\linewidth]{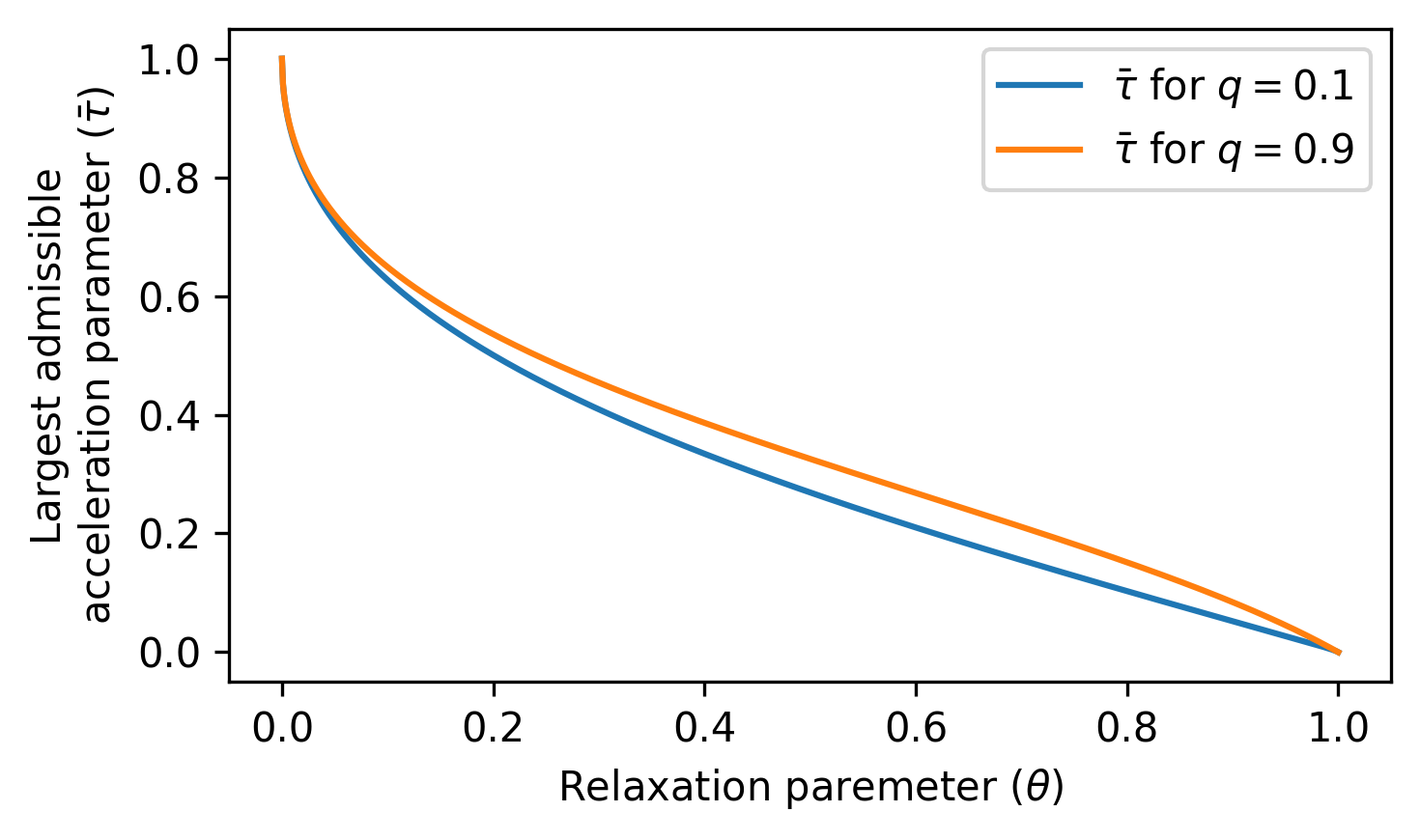}
    \includegraphics[width=0.48\linewidth]{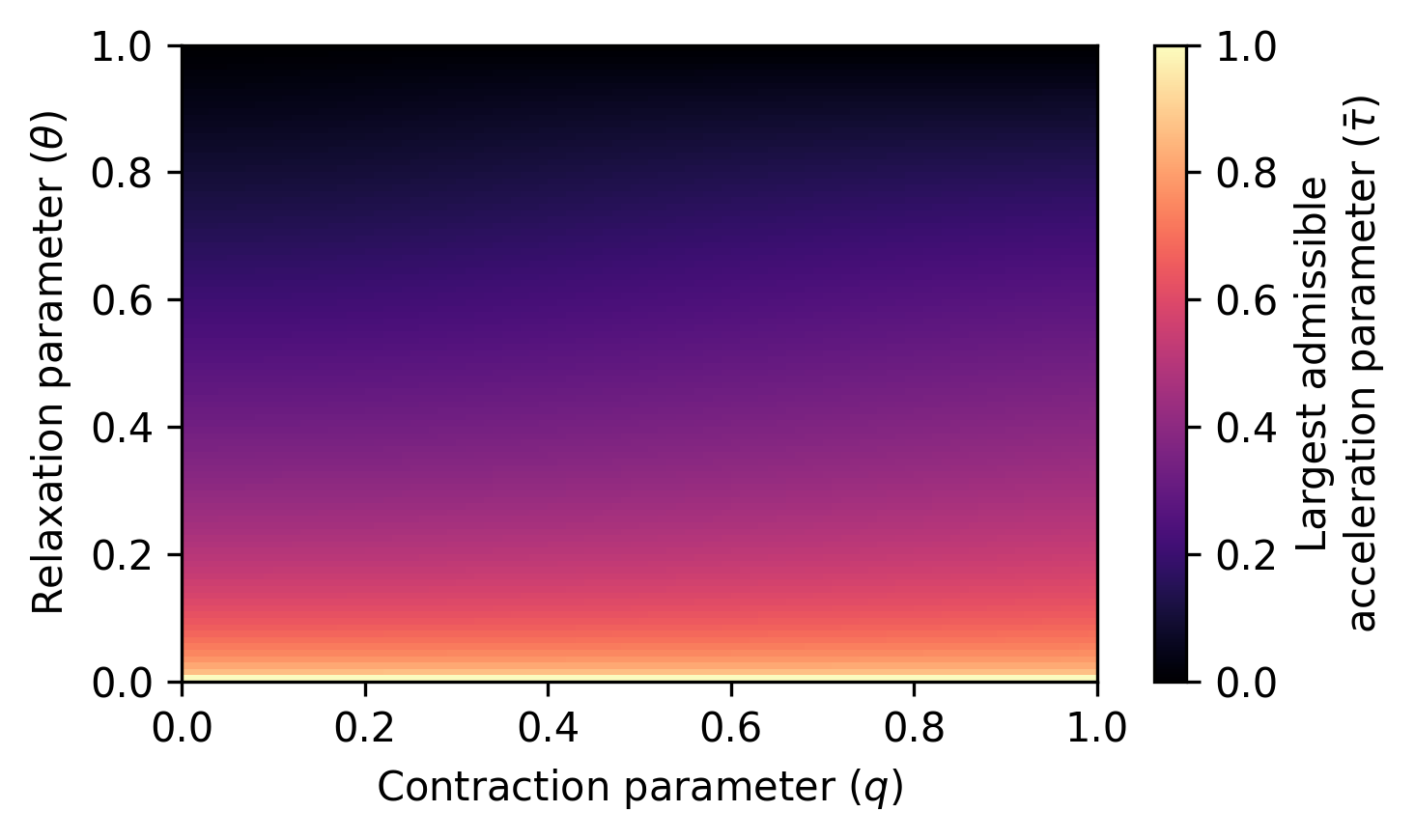}
    \caption{Largest acceleration parameter $\bar\tau$ as a function of the relaxation and the contraction parameters $\theta$ and $q$.}
    \label{fig:PA}
\end{figure}

\subsection{Convergence Analysis of Restarting Procedure}\label{S:conv}

For our complexity analysis, we introduce two gap functions, one associated with the solution set $\mathcal{S}_{0}$, and one with feasibility $\Zer(\opM)$.
\begin{definition}
The \textit{optimality gap} is defined as
\begin{equation}\label{eq:GapOpt}
    \GapOpt(u)\eqdef \sup_{v\in \scrS_{0}}\inner{ \opG v, u-v}.
\end{equation}
\end{definition}
We observe that if $u$ is feasible, namely if $u\in\scrS_{0}=\Zer(\opM)$, then the condition $\GapOpt(u)\le 0$ is equivalent to $u\in\Zer(\opG+\NC_{\scrS_{0}})$, such that $u$ solves Problem \eqref{eq:P}. It is important to note, however, that $\GapOpt(u)\le 0$ provides limited information when $u\notin \Zer(\opM)$. It is thus important to obtain a lower bound on the optimality gap. 
The feasibility gap is motivated by the observation that, by maximal monotonicity of $\opM$, for any $x\in \scrH$,
\begin{equation*}
    x\in \scrS_{0}=\Zer(\opM)\Leftrightarrow \inf_{y\in \dom(\opM)}\inf_{v\in \opM y}\inner{ 0-v, x-y} \geq 0.
\end{equation*}
\begin{definition}
The \emph{feasibility gap} $\GapFeas:\mathcal{H}\to \R$ is defined as
\begin{equation}\label{eq:GapFeas}
    \GapFeas(x)=\sup_{y\in \dom(\opM), v\in \opM y}\inner{ v, x-y}.
\end{equation}
\end{definition}
We note that $\GapFeas(x)\le 0$ implies $x\in\Zer(\opM)$, and that $\GapFeas(x)\ge 0$ holds for all $x\in \mathcal{H}$ since, by assumption, the problem admits at least one solution; in particular, $\Zer(\opM)\neq\emptyset$. Since $\GapOpt$ and $\GapFeas$ are pointwise suprema of affine functions, both are weakly lower semicontinuous and convex.

Our convergence rates in terms of the optimality gap require a lower bound on the a priori signless optimality gap function $\GapOpt$. We derive such a lower bound within a specific geometric setting, imposing a Hölderian error bound on the feasibility gap. 
\begin{definition}\label{def:WS}
We say that the zero set $\scrS_{0}=\Zer(\opM)$ is \textit{weakly sharp} with constant $\kappa>0$ and order $\rho\geq 1$ if, for all $u\in \dom(\opM)$,
\begin{equation}\label{eq:WS}
    \kappa\cdot \dist(u, \zer(\opM))^{\rho}\le \GapFeas(u).
\end{equation}
\end{definition}

    The notion of weak sharpness of zero sets is closely related to the classical concept of weak sharpness for variational inequalities \cite{marcotte_weak_1998}. Geometric conditions ensuring the error-bound characterization \eqref{eq:WS} can be found in \cite{Al-Homidan:2017aa}. The recent papers \cite{samadi_improved_2025} and \cite{Alvarez:2001aa} rely also heavily on this assumption. An important class of examples arises when $\opM$ is single-valued and $\rho=1$, notably in monotone linear complementarity problems in finite dimensions under nondegeneracy conditions \cite{burke_weak_1993,pang1997error}. 
    \begin{remark}
    In the case where the lower-level problem is determined by a convex optimization problem, weak-sharpness implies an Hölderian errror bound, an assumption already imposed by \cite{Cabot2005} in the context of hierarchical minimization. Specifically, let us assume that $\opM=\partial f$ the subgradient of a proper convex and lower semi-continuous function $f:\scrH\to\R\cup\{+\infty\}$. Then, 
    \[
    \sup_{v\in\partial f(y)}\inner{v,u-y}=f'(y;u-y)\leq f(u)-f(y),
    \]
    so that $\GapFeas(u)\leq f(u)-\min f$. It follows that Inequality \eqref{eq:WS} implies 
    \[
    \kappa\dist(u,\argmin f)^{\rho}\leq f(u)-\min f\qquad \forall u\in\scrH.
    \]
If $\rho=1$, this is the weak-sharpness condition of \cite{burke_weak_1993}. The case $\rho=2$ corresponds to the "quadratic growth" condition of \cite{drusvyatskiy2018error}. Rates of convergence under this Hölderian error bound assumption have recently been derived in \cite{boct2025accelerating}. 
\end{remark}
We next define the following quantities, which are finite because $\dom(\opM)$ is assumed to be bounded:
\begin{align*}
    D_\opM=\diam(\dom(\opM)), \qquad\qquad & C_\opM=\sup_{x\in \dom(\opM)}\sup_{v\in \opM x}\|v\|, \quad & C_\opG=\sup_{x\in \dom(\opM)}\opG(x).
\end{align*}
We also give more explicit expressions for the averaging coefficients $(\lambda_n)_{n}$ and the corresponding \textit{averaged iterates} which are produced by \texttt{DANTE}. They read as
\[
\lambda_{0}= 1,\text{ and for all }n\geq 1,\;\lambda_n=\prod_{i=0}^{n-1}\left(1+\frac{2\mu\beta_i}{\alpha}\right),\;\overline w_{N}=\frac{\sum_{n=0}^{N-1}\lambda_n\beta_n w_{n+1}}{\sum_{n=0}^{N-1}\lambda_n\beta_n}.
\]
In order to guarantee our convergence results, we must make some weak assumptions on the sequence of parameters.

\begin{assumption}\label{ass:params}
The parameter sequences satisfy the following conditions:
\begin{enumerate}

\item $\beta_{n}\in[0,\beta_{\max}]$ is non-increasing and $\lim_{n\to\infty}\beta_{n}=0$.
   \item The sequence $(\lambda_{n}\beta_{n})_{n\geq 0}$ satisfies $\sum_{n\geq 0}\lambda_{n}\beta_{n}=\infty$, and 
   \begin{equation}
\frac{\sum_{n=0}^{N-1}\lambda_{n}\beta^{2}_{n}}{\sum_{n=0}^{N-1}\lambda_{n}\beta_{n}}\to 0\text{ as }N\to\infty.       
   \end{equation}
    \item The sequence $(\ce_n)_{n\geq 0}$ satisfies $\ce_n=o(\beta_n)$. We define $\ce_{\max}=\max_{n\ge 1}\ce_n$. 
\end{enumerate}
\end{assumption}
Note that since $\ce_{n}=O(\eps_{n})$, Assumption \ref{ass:params} implies $\eps_{n}/\beta_{n}\to 0$ as $n\to\infty.$ This means that when designing concrete parameters for implementing our algorithm, we can tailor the inner-loop tolerance $\eps_{n}$ to the evolution of the Tikhonov parameter $\beta_{n}\to 0$. 

We are now ready to state the proposition that is central to our analysis.

\begin{proposition}\label{lemma:main}
For any $(x,v)\in\gr(\opM)$, it holds that 
\begin{equation}\label{eq:energy}
\lambda_n\langle v, w_{n+1}-x\rangle+\lambda_n\beta_n \langle \opG(x), w_{n+1}-x\rangle\le -\frac{\alpha\lambda_{n+1}}{2}\|w_{n+1}-x\|^2+\frac{\alpha\lambda_{n}}{2}\|w_{n}-x\|^2+C_1\lambda_n \ce_n,
\end{equation}
where 
\begin{equation}\label{eq:red:C1}
C_1\eqdef 2\alpha\left(1+\frac{\mu\beta_{\max}}\alpha\right)(\ce_{\max}+D_\opM)+\frac{C_\opM}{\alpha}+\frac{\beta_{\max}C_\opG}{\alpha}.
\end{equation}
\end{proposition}

\begin{proof}
First, expand
\begin{align*}
\frac12\|w_{n+1}-x\|^2-\frac12\|w_n-x\|^2
&= -\frac12\|w_{n+1}-w_n\|^2+\langle w_{n+1}-w_n, w_{n+1}-x\rangle \\
&\le \langle w_n-\bar u_{n+1}, x-\bar u_{n+1}\rangle + \langle w_{n+1}-w_n, w_{n+1}-\bar u_{n+1}\rangle \\
&\quad +\langle \bar u_{n+1}-w_{n+1}, x-\bar u_{n+1}\rangle.
\end{align*}
For the last two terms, we obtain the bounds
\begin{align*}
    \inner{w_{n+1}-w_n,w_{n+1}-\bar{u}_{n+1}}&\leq \left(\norm{w_{n+1}-\bar{u}_{n+1}}+\norm{\bar{u}_{n+1}-\bar{u}_n}+\norm{\bar{u}_n-w_n}\right)\cdot\norm{w_{n+1}-\bar{u}_{n+1}}\\
    &\leq (\ce_n+D_\opM+\ce_{n-1})\ce_n\\
    &\leq(D_{\opM}+2\ce_{max})\ce_n,
\end{align*}
and
\begin{equation*}
    \inner{\bar{u}_{n+1}-w_{n+1},\bar{u}_{n+1}-x}\leq\norm{\bar{u}_{n+1}-w_{n+1}}\cdot\norm{\bar{u}_{n+1}-x}\leq D_\opM\ce_n.
\end{equation*}
Combining these inequalities yields
\begin{equation*}
 \frac12\|w_{n+1}-x\|^2-\frac12\|w_n-x\|^2 \le \langle w_n-\bar u_{n+1}, x-\bar u_{n+1}\rangle + 2(D_{\opM}+\ce_{\max})\ce_n.  
\end{equation*}

As $-\beta_n\opG(\bar u_{n+1})-\alpha(\bar u_{n+1}-w_n)\in \opM(\bar u_{n+1})$ and $v\in\opM(x)$, it follows by monotonicity of $\opM$ and by $\mu$-monotonicity of $\opG$ that
\begin{align*}
\langle w_n-\bar u_{n+1}, x-\bar u_{n+1}\rangle 
& \le \frac1\alpha \langle v, x-\bar u_{n+1}\rangle +\frac{\beta_n}\alpha \langle \opG(\bar u_{n+1}), x-\bar u_{n+1}\rangle \\
& \le \frac1\alpha \langle v, x-\bar u_{n+1}\rangle +\frac{\beta_n}\alpha \langle \opG(x), x-\bar u_{n+1}\rangle - \frac{\mu\beta_n}{\alpha} \|x-\bar u_{n+1}\|^2.
\end{align*}
Now we note that 
\begin{align*}
\|x-\bar u_{n+1}\|^2
&= \|x-w_{n+1}\|^2+\|w_{n+1}-\bar u_{n+1}\|^2+2\langle x-w_{n+1}, w_{n+1}-\bar u_{n+1}\rangle \\
&\ge \|x-w_{n+1}\|^2-2(\|x-\bar u_{n+1}\|+\|\bar u_{n+1}-w_{n+1}\|)\cdot \|w_{n+1}-\bar u_{n+1}\| \\
&\ge \|x-w_{n+1}\|^2-2(D_\opM+\ce_{\max})\ce_n.
\end{align*}
Combining the previous three relationships yields
\begin{align*}
 \frac12\|w_{n+1}-x\|^2-\frac12\|w_n-x\|^2 &\le \frac1\alpha \langle v, x-\bar u_{n+1}\rangle +\frac{\beta_n}\alpha \langle \opG(x), x-\bar u_{n+1}\rangle -\frac{\mu\beta_n}{\alpha}\norm{x-w_{n+1}}^2\\
 &\quad+2\frac{\mu\beta_n}{\alpha}(D_\opM+\ce_{max})\ce_n+2(D_{\opM}+\ce_{\max})\ce_n  .
\end{align*}
Rearranging the terms and bounding $\beta_n\le \beta_{\max}$ gives
\begin{equation*}
\left(\frac12+\frac{\mu\beta_n}{\alpha}\right)\|w_{n+1}-x\|^2-\frac12\|w_n-x\|^2\le \frac1\alpha \langle v, x-\bar u_{n+1}\rangle +\frac{\beta_n}\alpha \langle \opG(x), x-\bar u_{n+1}\rangle+2\left(1+\frac{\mu\beta_{\max}}{\alpha}\right)(D_{\opM}+\ce_{\max})\ce_n.
\end{equation*}
After noting that 
\begin{align*}
 \langle v, x-\bar u_{n+1}\rangle = \langle v, x-w_{n+1}\rangle+\langle v, w_{n+1}-\bar u_{n+1}\rangle \le  \langle v, x-w_{n+1}\rangle+C_\opM\ce_n,
\end{align*}
and 
\begin{equation*}
\langle \opG(x), x-\bar u_{n+1}\rangle = \langle \opG(x), x-w_{n+1}\rangle + \langle \opG(x), w_{n+1}-\bar u_{n+1}\rangle\le \langle \opG(x), x-w_{n+1}\rangle+C_\opG\ce_n,
\end{equation*}
and multiplying both sides by $\lambda_n\alpha$, we obtain
\begin{align*}
    \lambda_n\inner{v,w_{n+1}-x}&+\lambda_n\beta_n\inner{\opG(x),w_{n+1}-x}\\ &\leq -\left(\frac12+\frac{\mu\beta_n}{\alpha}\right)\lambda_n\alpha\|w_{n+1}-x\|^2 +\frac{\lambda_n\alpha}{2}\norm{w_n-x}^2 \\ &\quad+\left(C_\opM+\beta_{max}C_\opG+2\left(1+\frac{\mu\beta_{\max}}{\alpha}\right)(\ce_{\max}+D_\opM)\alpha\right)\lambda_n\ce_n.
\end{align*}
The results follow by recalling the definition of $\lambda_n$.
\end{proof}
The lower bound in the preceding proposition enables us to deduce bounds on both the optimality and feasibility gaps, provided the remaining term can be controlled. The next two lemmas accomplish precisely this.
\begin{lemma}[Optimality]\label{thm:red:opt}
For all $N\ge 1$, it holds that 
    \[
	\GapOpt(\overline{w}_{N})\leq C_2\cdot \frac{1}{\sum_{n=0}^{N-1}\lambda_n\beta_n}+C_1\cdot \frac{\sum_{n=0}^{N-1}\lambda_n\ce_n}{\sum_{n=0}^{N-1}\lambda_n\beta_n},
    \]
    where $C_1$ is defined in eq. \eqref{eq:red:C1} and 
    \begin{equation}\label{eq:red:C2}
    	C_2 \eqdef \alpha(\ce_0^2+D_\opM^2).
    \end{equation}
   In particular, if Assumption \ref{ass:params} holds, all weak accumulation points $\tilde{w}$ of $(\overline{w}_n)$ satisfy $\GapOpt(\tilde{w})\leq 0$.
\end{lemma}
\begin{proof}
Select $x\in \Zer(\opM)$ and $v=0$ in \eqref{eq:energy}, and sum the result for $n=0, \ldots, N-1$ to obtain that
\begin{align*}
\sum_{n=0}^{N-1}\lambda_n\beta_n\langle \opG(x), w_{n+1}-x\rangle
& \le \frac{-\alpha \lambda_N}{2}\|w_N-x\|^2+\frac{\alpha\lambda_0}{2}\|w_0-x\|^2+C_1\cdot \sum_{n=0}^{N-1}\lambda_n\ce_n \\
& \le {\alpha\lambda_0}(\|w_0-\bar u_{0}\|^2+\|\bar u_{0}-x\|^2)+C_1\cdot \sum_{n=0}^{N-1}\lambda_n\ce_n \\
& \le {\alpha}(\ce_0^2+D_\opM^2)+C_1\cdot \sum_{n=0}^{N-1}\lambda_n\ce_n.
\end{align*}
After normalization, this yields
\begin{equation*}
\langle \opG(x), \bar w_{N}-x\rangle
\le C_2\cdot \frac{1}{\sum_{n=0}^{N-1}\lambda_n\beta_n}+C_1\cdot \frac{\sum_{n=0}^{N-1}\lambda_n\ce_n}{\sum_{n=0}^{N-1}\lambda_n\beta_n}.
\end{equation*}
The inequality follows by taking the supremum over $x\in \Zer(\opM)$. The convergence conclusion follows by weak sequential lower-semicontinuity of $\GapOpt$.
\end{proof}
\begin{lemma}[Feasibility]\label{thm:red:feas}
Assume $(\beta_n)$ non-increasing. For all $N\ge 1$, it holds that 
    \[
        0\leq \GapFeas(\overline{w}_{N})\leq C_3\cdot \frac{1}{\sum_{n=0}^{N-1}\lambda_n\beta_n}+C_4\cdot \frac{\sum_{n=0}^{N-1}\lambda_n\beta_n^2}{\sum_{n=0}^{N-1}\lambda_n\beta_n}+C_1\cdot \frac{\sum_{n=0}^{N-1}\lambda_n\beta_n\ce_n}{\sum_{n=0}^{N-1}\lambda_n\beta_n},
    \]
    where 
    \begin{equation}\label{eq:red:C34}
        C_3=\beta_0C_2,\quad C_4\eqdef C_\opG(\ce_{\max}+D_\opM)
    \end{equation}
    where $C_1$ and $C_2$ are defined in Equations \eqref{eq:red:C1} and \eqref{eq:red:C2}.
    In particular, if Assumption \ref{ass:params} holds, all weak cluster point $\tilde{w}$ of $(\overline{w}_n)$ satisfy $\GapFeas(\tilde{w})=0$, such that $\tilde w\in \Zer(\opM)$.
\end{lemma}
\begin{proof}
Multiply both sides of \eqref{eq:energy} by $\beta_n$, and recall that $\beta_{n+1}\le \beta_n$. Moreover, notice that
\begin{equation*}
    \inner{\opG(x),w_{n+1}-x}\geq-\norm{\opG(x)}\cdot\norm{w_{n+1}-x}\ge -\|\opG(x)\|\cdot(\|w_{n+1}-\bar u_{n+1}\| + \|\bar u_{n+1}-x\|)\geq-C_{\opG}(\ce_n+D_{\opM}).
\end{equation*}
It then follows
$-\langle \opG(x), w_{n+1}-x\rangle\le C_\opG(\ce_{\max}+D_\opM)$. Summing for $n=0, \ldots, N-1$ then yields
\begin{align*}
\sum_{n=0}^{N-1}\lambda_n\beta_n\langle v, w_{n+1}-x\rangle 
&\le -\frac{\alpha \lambda_N\beta_N}{2}\|w_N-x\|^2+\frac{\alpha \lambda_0\beta_0}{2}\|w_0-x\|^2+C_1\cdot \sum_{n=0}^{N-1}\lambda_n\beta_n\ce_n\\
&\quad +C_\opG(\ce_{\max}+D_\opM)\sum_{n=0}^{N-1}\lambda_n\beta_n^2 \\
&\le C_3+C_1\cdot \sum_{n=0}^{N-1}\lambda_n\beta_n\ce_n+C_4\cdot \sum_{n=0}^{N-1}\lambda_n\beta_n^2.
\end{align*}
By dividing by $\sum_{n=0}^{N-1}\lambda_n\beta_n$ and taking the supremum over points $x\in\scrS_{0}$ on both sides, we obtain the claimed inequality. The convergence conclusion again follows by weak sequential lower semicontinuity of $\GapFeas$.
\end{proof}

We may now combine the two previous lemmas to state our main theorem of this paper. We establish a combined complexity result in terms of the two introduced gap functions using the averaged trajectory $\bar{w}_{N}$.

\begin{theorem}[Main Theorem]\label{thm:red:main}
Let $(w_n)_{n=0}^{N-1}$ be generated through Algorithm \ref{algo:outer_loop}. Assume $(\beta_n)_{n\geq 0}$ is non-increasing. Then the following statements hold true: 
\begin{itemize}
    \item[(a)] For all $N\ge 1$, the optimality gap is bounded as 
    \begin{equation}\label{eq:Opt1}
        -C_\opG\cdot \dist(\overline{w}_{N}, \scrS_{0})\le \GapOpt(\overline{w}_{N})\leq C_2\cdot \frac{1}{\sum_{n=0}^{N-1}\lambda_n\beta_n}+C_1\cdot \frac{\sum_{n=0}^{N-1}\lambda_n\ce_n}{\sum_{n=0}^{N-1}\lambda_n\beta_n},
    \end{equation}
    where $C_1$ and $C_2$ are problem-specific constants defined in \eqref{eq:red:C1} and \eqref{eq:red:C2}.
    \item[(b)] For all $N\geq 1$, the feasibility gap is bounded as 
    \begin{equation}\label{eq:feas1}
        0\leq \GapFeas(\overline{w}_{N})\leq C_3\cdot \frac{1}{\sum_{n=0}^{N-1}\lambda_n\beta_n}+C_4\cdot \frac{\sum_{n=0}^{N-1}\lambda_n\beta_n^2}{\sum_{n=0}^{N-1}\lambda_n\beta_n}+C_1\cdot \frac{\sum_{n=0}^{N-1}\lambda_n\beta_n\ce_n}{\sum_{n=0}^{N-1}\lambda_n\beta_n},
    \end{equation}
    where $C_3$ and $C_4$ are constants defined in \eqref{eq:red:C34}. 
\item[(c)]     With Assumption \ref{ass:params} in place, all weak accumulation points of $(\overline{w}_{n})_{n\geq 1}$ are solutions to \eqref{eq:P}.
\item[(d)]    If, moreover, $\scrS_{0}$ is weakly sharp with constant $\kappa$ or order $\rho\geq 1$, then it also holds that
    \[
        \GapOpt(\overline{w}_{N})\ge -C_\opG \kappa^{-1/\rho}\cdot \left(C_3\cdot \frac{1}{\sum_{n=0}^{N-1}\lambda_n\beta_n}+C_4\cdot \frac{\sum_{n=0}^{N-1}\lambda_n\beta_n^2}{\sum_{n=0}^{N-1}\lambda_n\beta_n}+C_1\cdot \frac{\sum_{n=0}^{N-1}\lambda_n\beta_n\ce_n}{\sum_{n=0}^{N-1}\lambda_n\beta_n}\right)^{1/\rho}.
    \]
\end{itemize}
\end{theorem}
\begin{proof}
\begin{itemize}
    \item[(a)] The upper bound is just Lemma \ref{thm:red:opt}. For the lower bound, we invoke the Cauchy-Schwarz inequality $\inner{\opG(x),\bar{w}_{N}-x}\geq \norm{\opG(x)}\cdot\norm{\bar{w}_{N}-x}$, implying 
    \[
    \GapOpt(\bar{w}_{N})\geq - C_{\opG}\dist(\bar{w}_{N},\scrS_{0}).
    \]
    This proves the assertion.
    \item[(b)] This is just Lemma \ref{thm:red:feas}.
    \item[(c)] Combine Lemma \ref{thm:red:opt} with Lemma \ref{thm:red:feas}. 
    \item[(d)] If $\scrS_{0}$ is weakly sharp with constant $\kappa$ of order $\rho\geq 1$, then 
    \begin{align*}
    \GapOpt(\bar{w}_{N})&\geq -C_{\opG}\dist(\bar{w}_{N},\scrS_{0})\\
    &\geq -C_{\opG}\left[\frac{1}{\kappa}\GapFeas(\bar{w}_{n})\right]^{1/\rho}\\
    &\geq -\frac{C_{\opG}}{\kappa^{1/\rho}}\left(C_3\cdot \frac{1}{\sum_{n=0}^{N-1}\lambda_n\beta_n}+C_4\cdot \frac{\sum_{n=0}^{N-1}\lambda_n\beta_n^2}{\sum_{n=0}^{N-1}\lambda_n\beta_n}+C_1\cdot \frac{\sum_{n=0}^{N-1}\lambda_n\beta_n\ce_n}{\sum_{n=0}^{N-1}\lambda_n\beta_n}\right)^{1/\rho}.
        \end{align*}
\end{itemize}
\end{proof}

By selecting specific regularization parameters, we may obtain explicit rates on the convergent quantities, allowing direct comparisons to prior works. The choice of parameters differs between the solely monotone setting, detailed below, and the strongly monotone setting, detailed in Corollary \ref{coro:main2}.

\begin{corollary}\label{coro:main}
Let $(w_n)_{n=0}^{N}$ be generated through Algorithm \ref{algo:outer_loop}. Assume that $\opG$ is merely monotone (i.e. $\mu=0$). Let $\beta_n=(n+1)^{-b}$ for $b\in (0, 1)$ and choose $\eps_{n}$ such that $\ce_n=o(\beta_n)$. Then, for all $N\ge 2^{1/(1-b)}$, it holds that 
    \[
        \GapOpt(\overline{w}_{N})\leq \mathcal O\left(\frac1{N^{1-b}}\right),
    \]
	as well
    \[
        0\leq \GapFeas(\overline{w}_{N})\leq\begin{cases}\mathcal O\left(\frac1{N^{b}}\right) & \text{if $b\in (0, 0.5)$}\\
		\mathcal O\left(\frac{\ln(N)}{N^{1/2}}\right) & \text{if $b=0.5$} \\
		\mathcal O\left(\frac1{N^{1-b}}\right) & \text{if $b\in (0.5, 1)$}.
		\end{cases}
    \]

    If, moreover, $\scrS_{0}$ is weakly sharp with constant $\kappa>0$ of order $\rho\geq 1$, then it also holds that
    \[
        \GapOpt(\overline{w}_{N})\ge\mathcal O(\GapFeas(\overline w_N)^{1/\rho}).
    \]
\end{corollary}
\begin{proof}
	By standard integral bounds, one can show that for $c\in (0, 1)$ and $N\ge 2^{1/(1-b)}$,
	\[
	\frac{N^{1-c}}{2(1-c)}\leq\sum_{n=0}^{N-1}(n+1)^{-c}\le \frac{N^{1-c}}{1-c}.
	\]
	Specifically, $\sum_{n=0}^{N-1}\beta_n=\Theta(N^{1-b})$, and 
	\[
		\frac{\sum_{n=0}^{N-1}\beta_n^2}{\sum_{n=0}^{N-1}\beta_n}=\begin{cases}\Theta(N^{-b}) & \text{if $b\in (0, 0.5)$}\\
		\Theta(\ln(N)N^{-1/2}) & \text{if $b=0.5$} \\
		\Theta(N^{b-1}) & \text{if $b\in (0.5, 1)$}
		\end{cases}.
	\]
    
	As $\ce_n=o(\beta_n)$, the conclusions follow by Theorem \ref{thm:red:main}.
\end{proof}

\begin{corollary}\label{coro:main2}
Let $(w_n)_{n=0}^{N}$ be generated through Algorithm \ref{algo:outer_loop} when $\opG$ is strongly monotone (i.e. $\mu>0$). Let $\beta_n=\frac{\alpha}{2\mu n+\xi}$ for $\xi> 0$ and let $\ce_n=o(\beta_n)$. For all $N$ sufficiently large, it holds that 
    \[
        \GapOpt(\overline{w}_{N})\leq \mathcal O\left(\frac1{N}\right),
    \]
	and that 
    \[
        0\leq \GapFeas(\overline{w}_{N})\leq \mathcal O\left(\frac {\log(N)}N\right).
    \]
    If, moreover, $\scrS_{0}$ is weakly sharp with constant $\kappa$ of order $\rho\geq 1$, then it also holds that $\GapOpt(\overline{w}_{N})\ge\mathcal O(\GapFeas(\overline w_N)^{1/\rho}).$
\end{corollary}
\begin{proof}
Under the stated hypothesis, one easily computes that $\lambda_n=\frac{2\mu n+\xi}{\xi}$. As such, 
\[
\sum_{n=0}^{N-1}\lambda_n\beta_n= \frac{N\alpha}{\xi},
\]
and 
\[
\sum_{n=0}^{N-1}\lambda_n\beta_n^2=\frac{\alpha^2}{\xi}\cdot\sum_{n=0}^{N-1}\frac{1}{2\mu n+\xi}\le \frac{\alpha^2}{\xi}\cdot \left(\frac{1}{\xi}+\frac{1}{2\mu}\log\frac{2\mu N+\xi}{\xi}\right)=\mathcal O(\log(N)).
\]
The conclusion follows by applying these estimates to Theorem \ref{thm:red:main}.
\end{proof}

\subsection{Total Complexity Analysis}\label{ssec:comb}

We now finally perform the complete complexity analysis of our scheme \texttt{DANTE} by combining the bounds derived for Algorithms \ref{algo:inner_loop} and \ref{algo:outer_loop}. Given a sequences of user-defined tolerances $(\eps_{n})_{n=0}^{N}$, we known that Algorithm \ref{algo:inner_loop} terminates after at most $\bar{K}_{\eps_{n}}\equiv \bar{K}_{n}$ iterations (Corollary \ref{coro:complexity}, eq. \eqref{eq:Kn}). To assess the total iteration complexity of Algorithm \ref{algo:outer_loop} for a pre-defined number of restarts $N\geq 1$, we thus need to compute the \emph{total iteration complexity}
\begin{equation}\label{eq:TC}
\mathcal{C}(N,(\eps_{n})_{n=0}^{N})\eqdef \sum_{n=0}^{N-1}\K(\eps_{n}).
\end{equation}
To estimate the total iteration complexity, we make one final assumption.

\begin{assumption}\label{ass:add2}
    We assume $\opZ_n$ is domain-forward on $\dom(\opM)$: 
    \[
    \opZ_n(\dom(\opM))\subset \dom(\opM).
    \]
    Finally, we assume that $\theta_{n, k}$ is uniformly lower bounded by a positive quantity, such that $\overline Q_n\le \overline Q<1$.
\end{assumption}

\begin{remark}
   The domain-forwardness assumption may appear restrictive on the first glance. However, for all relevant fixed-point mappings from Section \ref{sec:splitting}, the fixed-point transportation map will take the form of a resolvent operator for which the requirement is naturally fulfilled due to maximal monotonicity of the operator.
\end{remark}

\begin{theorem}\label{thm:comb}
    Let Assumptions \ref{ass:1}-\ref{ass:add2} hold. Given $N\geq 1$, and inner loop tolerance sequence $(\eps_{n})_{n=0}^{N}$, we have 
        \[
       \mathcal{C}(N,(\eps_{n})_{n=0}^{N})\leq\frac{N(2\log(C)+\log(\bar Q))+2\sum_{n=1}^N\log(1/\varepsilon_n)}{\log(1/\bar Q)},
    \]
    where the constant $C$ is defined in \eqref{eq:C} and $\bar{Q}$ is defined in Assumption \ref{ass:KM}. 
    
    Specifically, if $\varepsilon_n=(n+1)^{-\eta}$ and $\beta_{n}=(n+1)^{-b}$ where $\eta\geq b$, then $\mathcal{C}(N,(\eps_{n})_{n=0}^{N})={\mathcal O}(N\log N)$. 
\end{theorem}
\begin{proof}
    By Corollary \ref{coro:complexity}, for each restarting epoch $n\in\{0,\ldots,N-1\}$ the total number of iterations of the inner loop is bounded by 
    \[
       \K(\eps_{n})\leq \bar{K}_{n}=\left\lceil\frac{2\log(C)+2\log(1/\varepsilon_n)}{\log(1/\bar Q)}\right\rceil\le \frac{2\log(C)+2\log(1/\varepsilon_n)}{\log(1/\bar Q)}+1.
    \]
   Hence, 
    \[
\sum_{n=0}^{N-1}\bar{K}_n\le \frac{2N\log(C)+2\sum_{n=1}^N\log(1/\varepsilon_n)}{\log(1/\bar Q)}+N=\mathcal {O}\left(N+\sum_{n=0}^{N-1}\log(1/\varepsilon_n)\right)
    \]
    gives an upper bound on the total number of iterations for Algorithm \ref{algo:outer_loop} for $N\geq 1$ restarts. By setting $\varepsilon_n=(n+1)^{-\eta}$, we obtain that $\sum_{n=0}^{N-1}\log(1/\varepsilon_n)\le \eta N\log(N)$. 
\end{proof}

 Theorem \ref{thm:comb} shows that accounting for the total number of inner iterations, rather than only the number of outer iterations, introduces merely a logarithmic factor. This is expected, as double-loop schemes inherently incur such a term. Consequently, the complexities stated in Corollaries \ref{coro:main} and \ref{coro:main2} for the gap functions remain intact, up to logarithmic factors, when accounting for the total number of inner iterations.

\section{Relevant Splitting Methods}\label{sec:splitting}

Depending on the structure of $\opM$ in Problem \eqref{eq:P}, the auxiliary problem \eqref{eq:Aux-P} inherits different structural assumptions. The fixed-point encodings $\opT^{(w,\beta)}_{k}$ must be readily available for the numerical computations, and hence should exploit known problem structure. In Section \ref{ssec:2}, we analyze the case where $\opM$ may be written as the sum of two operators, and in Section \ref{ssec:3}, we consider a three-operator splitting scheme.

\begin{remark}
  With the following examples, it will become clear that the fixed-point encoding $\opT^{(w,\beta)}_{k}$ not only depends on the problem parameters $\alpha$ and $\beta$ but also on step size sequences $(\gamma_k)_{k}$. Therefore, we treat the step size policy as an ingredient of the fixed-point encoding strategy and consequently incorporate it in the map $\opT^{(w,\beta)}_{k}$, rather than treating it as a parameter of the inner loop presented in Algorithm \ref{algo:inner_loop}.
\end{remark}

\subsection{Two-Operator Splitting Schemes}\label{sec:FB}\label{ssec:2}

In this section, we focus on problems of the form \eqref{eq:P} with $\opM=\opA+\opF$, namely
\begin{equation}\label{eq:PFB}\tag{P-2-Split}
\VI(\opG,\scrS_{0}),\quad\text{where }\scrS_{0}\eqdef \zer(\opA+\opF). 
\end{equation}
We refine the problem assumptions through the following assumption:
\begin{assumption}\label{ass:FB}
We assume the following on Problem \eqref{eq:PFB}:
\begin{enumerate}
\item The lower-level solution set $\scrS_{0}$ is nonempty;
\item $\opG:\scrH\to\scrH$ is monotone and $L_\opG$-Lipschitz continuous; 
\item $\opF \colon \scrH\to \scrH$ is monotone and $L_\opF$-Lipschitz continuous;
\item $\opA\colon\scrH\to2^{\scrH}$ is maximally monotone with bounded domain.
\end{enumerate}
\end{assumption}
These assumptions imply that $\opM=\opA+\opF$ is maximally monotone with $\dom(\opM)=\dom(\opA)$. 

Given the parameters $\alpha$ and $\beta$ we define the function $\Phi^w_{\alpha,\beta}:\scrH\to\scrH$ by  
\begin{equation*}\label{eq:Phi}
\Phi^w_{\alpha,\beta}(v)\eqdef\opF(v)+\beta\opG(v)+\alpha(v-w),
\end{equation*}
which is Lipschitz continuous with constant $L_{\alpha, \beta}\eqdef L_\opF+\beta L_\opG+\alpha$ and $\alpha$-strongly monotone. Adopting this notation, we can rewrite the auxiliary problem \eqref{eq:Aux-P} as determining the unique element of $\zer(\opA+\Phi_{\alpha,\beta}^w)$. 

\paragraph{Forward-backward splitting.} 
The forward-backward splitting \cite{lions_splitting_1979,passty_ergodic_1979} is a very popular numerical scheme, defined in terms of the fixed-point encoding map $\opT^{(w,\beta)}_\gamma\colon \scrH\to\scrH$ given by 
\begin{equation}\label{eq:FB}\tag{FB}
\opT^{(w,\beta)}_\gamma = \resolvent_{\gamma\opA}\circ (\Id - \gamma \Phi_{\alpha, \beta}^w).
\end{equation}
We note that $\Fix(\opT^{(w,\beta)}_\gamma )=\Zer(\opA+\Phi_{\alpha,\beta}^w)$, so $\opZ=\Id$ is the fixed-point transportation map. This is the setting studied in \cite{marschner_tikhonov_2025}.

\begin{lemma}\label{lem:contraction}
Assume Assumption \ref{ass:FB} holds and let $\opT^{(w,\beta)}_\gamma$ be defined as \eqref{eq:FB}. $\opT^{(w,\beta)}_\gamma$ is $q$-Lipschitz where 
\[
q\eqdef \sqrt{1-\gamma(2\alpha-\gamma\Lip^{2}_{\alpha,\beta})},
\]
uniformly over $w\in\scrH$. Hence, $\opT^{(w,\beta)}_\gamma$ is a contraction when $\gamma\in(0,2\alpha/\Lip^{2}_{\alpha,\beta})$.
\end{lemma}
\begin{proof}
Pick $z_{1},z_{2}\in \scrH$ and set $T=T^{(w,\beta)}_{\gamma}$. By nonexpansiveness of $\resolvent_{\gamma\opA}$, and $\Lip_{\alpha,\beta}$-Lipschitz continuity and $\alpha$-strong monotonicity of $\Phi_{\alpha,\beta}$, 
\begin{align*}
\norm{T(z_1)-T(z_2)}^{2} 
&= \norm{\resolvent_{\gamma\opA}(z_{1}-\gamma\Phi_{\alpha,\beta}(z_{1},w))-\resolvent_{\gamma\opA}(z_{2}-\gamma\Phi_{\alpha,\beta}(z_{2},w))}^{2}\\
&\leq \norm{z_{1}-z_{2}+\gamma(\Phi_{\alpha,\beta}(z_{2},w)-\Phi_{\alpha,\beta}(z_{2},w))}^{2}\\
&=\norm{z_{1}-z_{2}}^{2}+2\inner{z_{1}-z_{2},\Phi_{\alpha,\beta}(z_{2},w)-\Phi_{\alpha,\beta}(z_{1},w)}\\
&\qquad +\gamma^{2}\norm{\Phi_{\alpha,\beta}(z_{1},w)-\Phi_{\alpha,\beta}(z_{2},w)}^{2}\\
&\leq\norm{z_{1}-z_{2}}^{2}-2\gamma\alpha\norm{z_{1}-z_{2}}^{2}+\Lip^{2}_{\alpha,\beta}\gamma^{2}\norm{z_{1}-z_{2}}^{2}\\
&=(1-2\gamma\alpha+\Lip^{2}_{\alpha,\beta}\gamma^{2})\cdot \norm{z_{1}-z_{2}}^{2}.
\end{align*}
The contraction property follows by choosing $\gamma$ so that $\gamma(2\alpha-\gamma\Lip^{2}_{\alpha,\beta})\in(0,1)$.
\end{proof}

\paragraph{Backward-forward splitting.}
Similarly, we can employ the backward-forward splitting \cite{attouch_backward_2018} for the auxiliary problem. Consider the operator 
\begin{equation}\label{eq:BF}\tag{BF}
\opT^{(w,\beta)}_\gamma= (\Id-\gamma \Phi_{\alpha, \beta}^w)\circ \resolvent_{\gamma\opA}.
\end{equation}
Note that $\zer(\opA+\Phi_{\alpha,\beta}^w)=\resolvent_{\gamma\opA}(\text{Fix}(\opT^{(w,\beta)}_\gamma))$, which means that the last iterate of the inner loop scheme needs one more evaluation of the resolvent of the operator $\opA$ to get close to the unique zero of $\opA+\Phi_{\alpha,\beta}^w$. Specifically, the fixed-point transportation map is $\opZ=\resolvent_{\gamma \opA}$.
\begin{lemma}\label{lem:contractionBF}
Assume Assumption \ref{ass:FB} holds, and let $\opT_\gamma$ be the fixed-point encoding defined in \eqref{eq:BF}. It holds that $\opT_\gamma$ is $q$-Lipschitz where 
\[
    q=\sqrt{1-\gamma(2\alpha-\gamma\Lip^{2}_{\alpha,\beta})},
\]
uniformly over $w\in\scrH$. In particular, $\opT^{(w,\beta)}_\gamma$ is a contraction when $\gamma\in(0,2\alpha/\Lip^{2}_{\alpha,\beta})$.
\end{lemma}
\begin{proof}
    The proof is similar to the one of Lemma \ref{lem:contraction}.
\end{proof}

\paragraph{Douglas-Rachford splitting.} Analogously, we can also consider the Douglas-Rachford splitting \cite{lions_splitting_1979}, whose operator $\opT$ is given by
\begin{equation}\label{eq:DR}\tag{DR}
\opT^{(w,\beta)}= \frac12(\Id+R_{\Phi^{w}_{\alpha,\beta}} \circ R_\opA).
\end{equation}
where $R_\bullet$ is the reflected resolvent given by $R_\bullet=2\resolvent_{\bullet}-\Id$. Note that $\resolvent_\opA(\Fix(\opT))=\Zer(\opA+\Phi_{\alpha, \beta}^w)$, such that $\opZ=\resolvent_\opA$ takes over the role of the fixed-point transportation map.

\begin{lemma}{\cite[Corollary 3.1]{moursi_douglas_2019}}\label{lem:contractionDR}
Assume Assumption \ref{ass:FB} holds and let $\opT$ be the fixed-point encoding defined in \eqref{eq:DR}. It holds that $\opT^{(w,\beta)}$ is $q$-Lipschitz where 
\[
    q=\frac12+\frac12\left(\frac{1-2\alpha+ L_{\alpha, \beta}^2}{1+2 \alpha+ L_{\alpha, \beta}^2}\right)^{1/2},
\]
uniformly over $w\in\scrH$. In particular, $\opT^{(w,\beta)}$ is a contraction.
\end{lemma}

\subsection{Three-Operator Splitting Scheme}\label{ssec:3}

In this section we focus on problems of the form \eqref{eq:P} with $\opM=\opF+\opA+\opB$, namely
\begin{equation}\label{eq:PTOS}\tag{P-3-Split}
\VI(\opG,\scrS_{0}),\quad\text{where }\scrS_{0}\eqdef \zer(\opF+\opA+\opB). 
\end{equation}
We refine the problem assumptions through the following assumption:
\begin{assumption}\label{ass:TOS}
We assume the following on Problem \eqref{eq:PTOS}:
\begin{enumerate}
\item the set $\scrS_{0}$ is nonempty;
\item $\opG:\scrH\to\scrH$ is monotone and $L_\opG$-Lipschitz continuous; 
\item $\opF \colon \scrH\to \scrH$ is $L_\opF$-Lipschitz continuous;
\item $\opA, \opB\colon\scrH\to2^{\scrH}$ are maximally monotone such that $\opA$ has bounded domain satisfying $\dom(\opA)\subset \dom(\opB)$.
\end{enumerate}
\end{assumption}
Under these assumptions, $\dom(\opM)=\dom(\opA)$. We note that in most practical examples, $\dom(\opB)=\scrH$, such that $\dom(\opA)\subset \dom(\opB)$ holds naturally.

We consider the auxiliary function $\Psi^w_{\alpha,\beta}:\scrH\to\scrH$ given by 
\begin{equation*}\label{eq:Psi}
\Psi^w_{\alpha,\beta}(v):=\opF(v)+\beta\opG(v)+\alpha(v-w),
\end{equation*}
which is $\alpha$-strongly monotone and $L_{\alpha, \beta}
\eqdef L_\opF+\beta L_\opG+\alpha$-Lipschitz continuous, and hence $\nu_{\alpha, \beta}=\alpha/L_{\alpha, \beta}^2$-cocoercive. The auxiliary Problem \eqref{eq:Aux-P} reads as $\zer(\opA+\opB+\Psi^w_{\alpha,\beta})$, and may be solved through a three-operator splitting scheme \cite{davis_threeoperator_2017a}, namely through the fixed-point encoding map $\opT^{(w,\beta)}_\gamma\colon \scrH\to \scrH$ depending on a step-size $\gamma$ given by
\begin{equation}\label{eq:3op}\tag{TOS}
\opT^{(w,\beta)}_\gamma=\Id-\resolvent_{\gamma\opB}+\resolvent_{\gamma\opA}\circ(2\resolvent_{\gamma\opB}-\Id-\gamma \Psi^w_{\alpha,\beta}\circ\resolvent_{\gamma\opB}).
\end{equation}
We know that $\resolvent_{\gamma \opA}(\Fix(\opT))=\Zer(\opA+\opB+\Psi_{\alpha, \beta}^w)$, such that $\opZ=\resolvent_{\gamma \opA}$ is the fixed-point transportation map. The operator $\opT^{(w,\beta)}_\gamma$ is non-expansive under the provided assumptions, and a contraction for well-chosen $\gamma$.
\begin{lemma}{\cite[Theorem D.6]{davis_threeoperator_2017a}}\label{lem:TOS}
Let $\opT^{(w,\beta)}_\gamma$ be the fixed-point encoding defined in \eqref{eq:3op}. It holds that $\opT^{(w,\beta)}_\gamma$ is non-expansive. Moreover, if $\opB$ is $L_\opB$-Lipschitz continuous and $\gamma<\eta \nu_{\alpha, \beta}$, then, uniformly over $w\in\scrH$, $\opT^{(w,\beta)}_\gamma$ is $\sqrt{1-q}$-Lipschitz with constant 
    \[
        q=\frac{2\gamma\alpha(1-\eta)}{(1+\gamma L_B)^2},
    \]
    where $\eta\in [0, 1]$ is arbitrary. In particular, $\opT^{(w,\beta)}_\gamma$ is a contraction when 
    \[
        \gamma\in\left(0,\frac{2\alpha(1-\eta)-2L_\opB+\sqrt{4\alpha^2(1-\eta^2)-8\alpha(1-\eta)L_\opB}}{2L_\opB^2}\right).
    \]
\end{lemma}

\begin{remark}\label{rem:TOS_nonexpansive}
    Although we require $\opB$ to be Lipschitz continuous for Assumption \ref{ass:contraction} to be satisfied, we shall allow it not to be in our implementations, as per Remark \ref{rem:nonexpansive}. This is specifically the case in Section \ref{sec:image}.
\end{remark}


\section{Numerical Experiments}\label{sec:numerics}

In this section, we develop numerical examples showcasing the versatility and applicability of Algorithm \ref{algo:inner_loop}. Specifically, we consider an equilibrium problem in Section \ref{sec:equi}, a least-norm least-squares problem in Section \ref{sec:lst}, and an image inpainting problem in Section \ref{sec:image}.

All the code may be found on the author's GitHub page\footnote{See \url{https://github.com/Hierarchical-VIs/Regularization-Methods-for-HVIs}.}, and is run on Intel Xeon Platinum 8380 CPUs.

\subsection{Equilibrium selection}\label{sec:equi}

We consider the two-player zero-sum game from \cite{samadi_improved_2025}, given by
\[
    \begin{dcases}
        &\min_{x_1}f_1(x_1, x_2)\eqdef 20-0.1x_1x_2+x_1 \\
        &\text{s.t.}~ x_1\in X_1\eqdef [11, 60]
    \end{dcases}\qquad 
    \begin{dcases}
        &\min_{x_2}f_2(x_1, x_2)\eqdef -20+0.1x_1x_2-x_1 \\
        &\text{s.t.}~ x_2\in X_2\eqdef [10, 50].
    \end{dcases}
\]
Specifically, we seek a saddle point $(x_1^*, x_2^*)\in X\eqdef X_1\times X_2$ of the function $f(x_1, x_2)\eqdef 20-0.1 x_1x_2+x_1$, namely a point that satisfies
\[
    f(x_1^*, x_2)\le f(x_1^*, x_2^*)\le f(x_1, x_2^*)\quad \text{for all $(x_1, x_2)\in X$.}
\]
The solution set is characterized through the inclusion 
\[
    0\in \opF(x_{1},x_{2})+\NC_{X}(x_{1},x_{2})\eqdef \begin{bmatrix}
        \nabla_{x_1} f(x_1, x_2) \\
        -\nabla_{x_2} f(x_1, x_2)
    \end{bmatrix}+\NC_{X}(x_{1},x_{2})=\begin{bmatrix}
        0 & -0.1 \\ 0.1 & 0
    \end{bmatrix}\begin{bmatrix}
        x_1 \\ x_2
    \end{bmatrix} + \begin{bmatrix}
        1 \\ 0
    \end{bmatrix} + \NC_{X}(x_{1},x_{2})
\]
Analytically, we know the set of solutions to be $X_1\times \{10\}$. We attempt to find the best solution, namely to solve the problem 
\[
    \min \left\{\phi(x_1, x_2)\colon (x_1, x_2)\in \Zer(\opF+\NC_X)\right\},
\]
where $\phi(x)=\tfrac12\|x\|^2$, whose analytical solution is $(11, 10)$.
Equivalently, we aim to find a point $(x_1^*, x_2^*)\in \Zer(\opF+NC_X)$ such that 
\[
    \langle \nabla \phi (x_1^*, x_2^*), (w_1, w_2) - (x_1^*, x_2^*)\rangle \ge 0\quad \text{for all $(w_1, w_2)\in \Zer(F+\NC_X)$}.
\]
This matches Problem \eqref{eq:P} with $\opG=\nabla \phi$ for the upper-level operator and $\opM=\opF+\NC_X$ for the lower-level operator.

\paragraph{Experimental Setup.} For all the experiments, we consider the relaxation parameter of the inner loop $\theta_k\equiv \theta=0.7$ to be constant. Across inner loops, we consider the acceleration parameter $\tau_k\equiv \tau$ to be constant and to be the largest value satisfying Equation \eqref{eq:parameterineq}. We assume $\beta_n=(n+1)^{-\eta}$, where $\eta=0.55$, and that $\varepsilon_n=\bar\varepsilon\cdot {(n+1)^{-2}}$, where $\bar\varepsilon=10^{-3}$. We run a total of $1000$ iterations.

\paragraph{Results.}
Figure \ref{fig:EQ_Iterates} shows the averaged iterates $(\overline w_n)$, using forward-backward for the auxiliary problem, for various starting points, along with the feasible region $X$. Recall that, although we call it the feasible region, we have no a priori guarantee that the iterates remain within said region. 
\begin{figure}[H]
    \centering
    \includegraphics[width=1\linewidth]{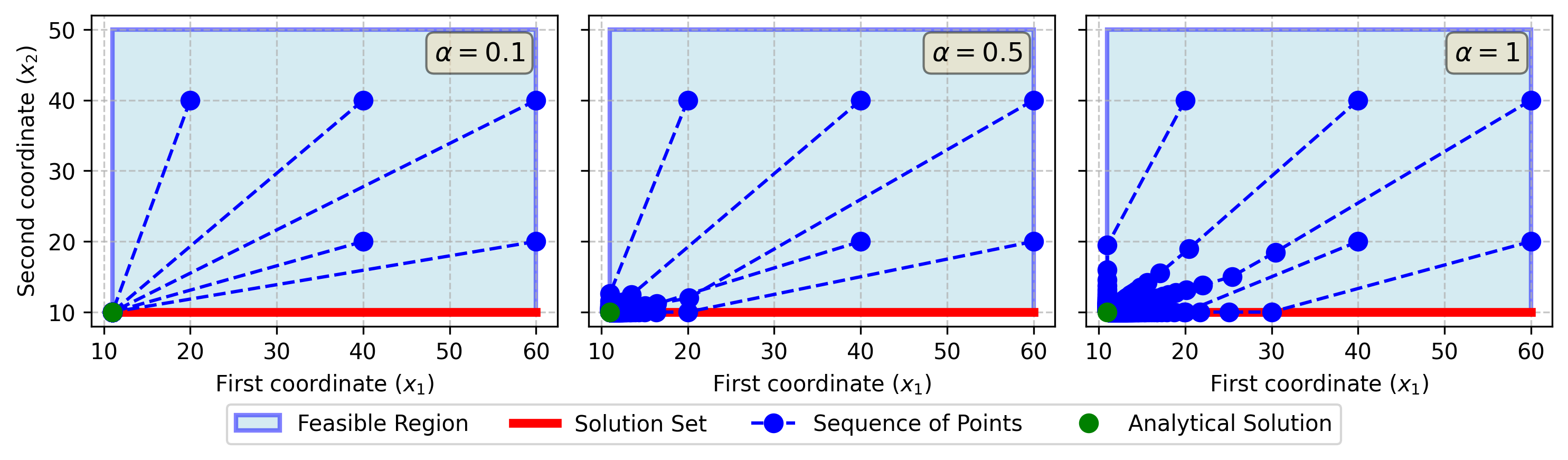}
    \caption{Average iterates for various initial points and proximal parameters.}
    \label{fig:EQ_Iterates}
\end{figure}
Moreover, we note that the optimality and feasibility gaps may be computed through a small subproblem in this simple example. Figures \ref{fig:EQ_Opt} and \ref{fig:EQ_Feas} plot these values for the different initial points. Although we observe differences, we note that all curves are decreasing to $0$. We note the importance of the tuning of the proximal parameter $\alpha$. In the present example, a smaller value yields faster convergence. As will be evidenced in later simulations, this is not always the case.

\begin{figure}[H]
    \centering
    \includegraphics[width=1\linewidth]{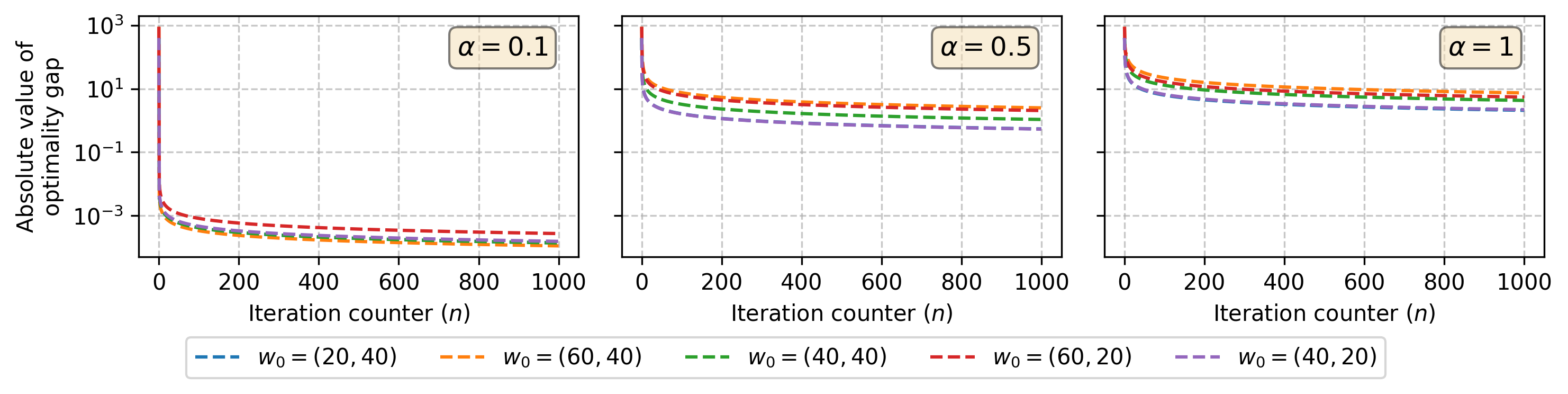}
    \caption{Optimality gap $\GapOpt(\overline w_n)$ for various initial points and proximal parameters.}
    \label{fig:EQ_Opt}
\end{figure}
\begin{figure}[H]
    \centering
    \includegraphics[width=1\linewidth]{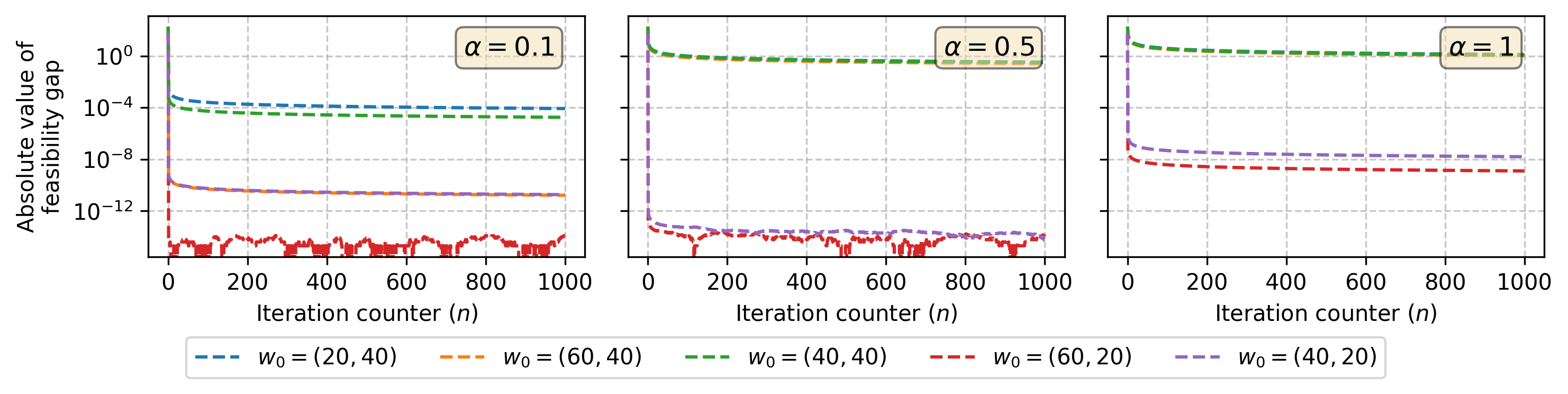}
    \caption{Feasibility gap $\GapFeas(\overline w_n)$ for various initial points and proximal parameters.}
    \label{fig:EQ_Feas}
\end{figure}

In Figures \ref{fig:EQ_Comp_Opt} and \ref{fig:EQ_Comp_Feas}, we compare various fixed-point encodings of the auxiliary problem. Specifically, we implement the methods from Section \ref{sec:FB}, namely forward-backward (FB), backward-forward (BF) and Douglas-Rachford (DR). We set $\bar\varepsilon=1$ and $\alpha=0.1$. We observe different behaviors for the different algorithms, but do not focus further on these in this work.

\begin{figure}[H]
    \centering
    \includegraphics[width=1\linewidth]{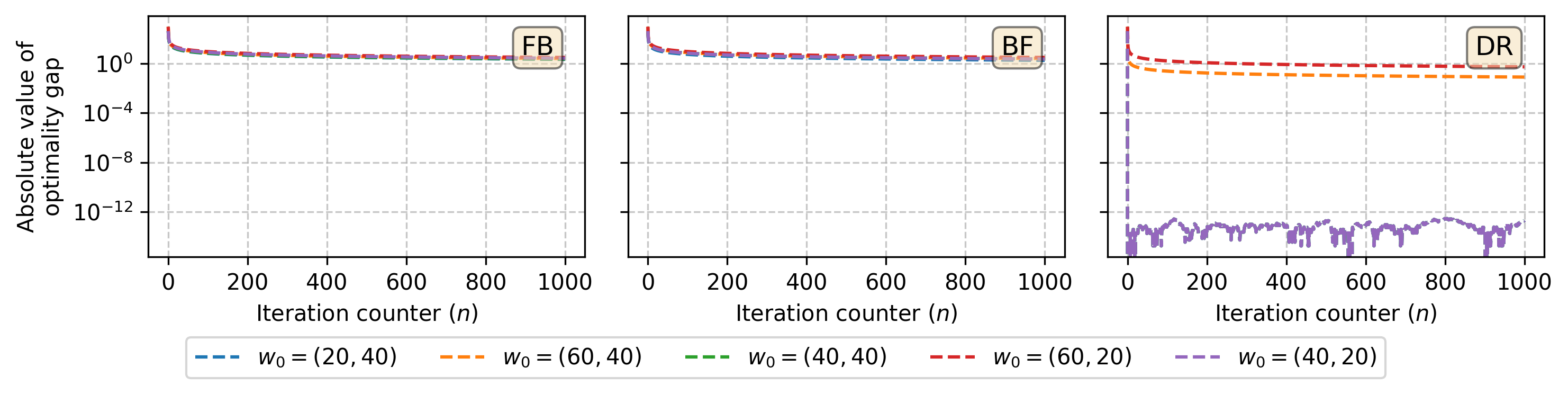}
    \caption{Optimality gap $\GapOpt(\overline w_n)$ for various auxiliary problem fixed-point encodings.}
    \label{fig:EQ_Comp_Opt}
\end{figure}

\begin{figure}[H]
    \centering
    \includegraphics[width=1\linewidth]{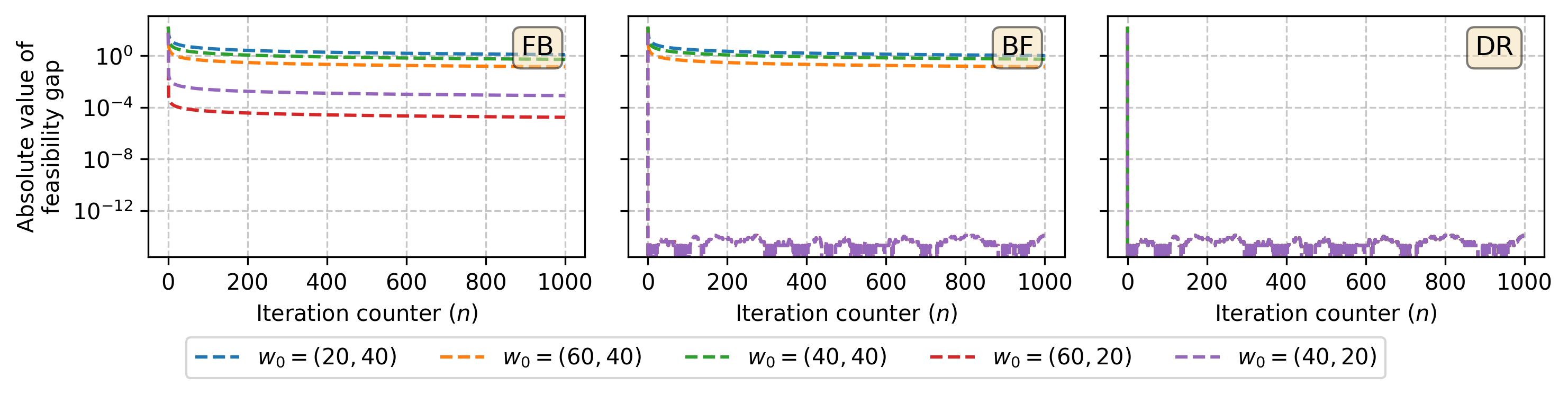}
    \caption{Feasibility gap $\GapFeas(\overline w_n)$ for various auxiliary problem fixed-point encodings.}
    \label{fig:EQ_Comp_Feas}
\end{figure}

\subsection{Least-Norm Least-Squares}\label{sec:lst}

Inspired by \cite{staudigl2023random}, for a matrix $A\in \R^{P\times Q}$ and a vector $b\in \R^P$, we consider the bilevel optimization problem given by
\begin{equation*}\label{eq:Experiment}
\min_{u\in \R^{Q}} \left\{\tfrac{1}{2}\|u\|^2~\colon ~u\in\argmin_{v\in X}\left\{\tfrac12\norm{Av-b}^2\right\}\right\},
\end{equation*}
where $X=[-1000, 1000]^{Q}$. We can reformulate this as a hierarchical variational inequality problem of the form \eqref{eq:PFB} with $\opG(u)=u$, $\opA=\partial\delta_{X}=\NC_{X}$, and $\opF(v)= 2A^{\top} (A v - b)$. We will make use of the forward-backward operator for the inner loop iterations.

\paragraph{Problem Setup.} We test the algorithm on four randomly generated instances with different dimensions given by 
$(P,Q) \in \{(70, 100), (100, 200), (100, 500), (300, 500)\}$. For each instance, the matrix $A$ is generated as a low-rank (and sparse) matrix via $A = U_1\cdot U_2$,
where $U_1 \in \mathbb{R}^{P \times R}$ and $U_2 \in \R^{R \times Q}$ are chosen such that each component follows a standard normal law, with $R=50$ to induce a low-rank structure. Moreover, to ensure the eigenvalues are not too large such that the step-size is not too small, we clip each singular value to the interval $[0, 10]$. A vector $s \in \R^Q$ is randomly generated with $20$ non-zero entries according to a uniform law on $[0, 10]$, and the observation vector $b$ is constructed as $b = A s + \nu$, where $\nu$ is a normal random vector with small entries, namely a standard normal vector scaled by a factor $0.1$. During the random generation, we ensure that $z=A^{\dagger}b\in X$, such that the analytical solution is known to be $z$, where $A^\dagger$ is the Moore-Penrose inverse.

\paragraph{Experimental Setup.} For the inner loop, we set $\theta_k\equiv \theta=0.75$ to be constant. Across inner loops, we consider $\tau_k\equiv \tau$ to be constant and to be the largest value satisfying Equation \eqref{eq:parameterineq}. We set $\beta_n=(n+1)^{-\eta}$, where $\eta=0.55$, and $\varepsilon_n=\overline \varepsilon \cdot (n+1)^{-1}$, where $\overline \varepsilon=\alpha\cdot 10^{-3}$. We run a total of $2000$ iterations, with an initial point $w_0$ randomly generated according to a standard normal distribution scaled by a factor $0.1$.

\paragraph{Results.} Figures \ref{fig:LS_Dist} and \ref{fig:LS_Func} show the evolution of the error sequence $\|w_n-z\|$ and the lower-level function gap sequence $\tfrac12\|Aw_n-b\|^2-\tfrac12\|Az-b\|^2$ across iterations. As expected, we observe a decrease for each problem. We note that the initial decrease for the error sequence is consequential and stagnates for a number of iterations afterward. In fact, for $\alpha=0.1$, the prescribed $2000$ iterations are not enough to observe a second decrease phase, whereas it is sufficient for $\alpha\in \{1, 10\}$.

\begin{figure}[H]
    \centering
    \includegraphics[width=1\linewidth]{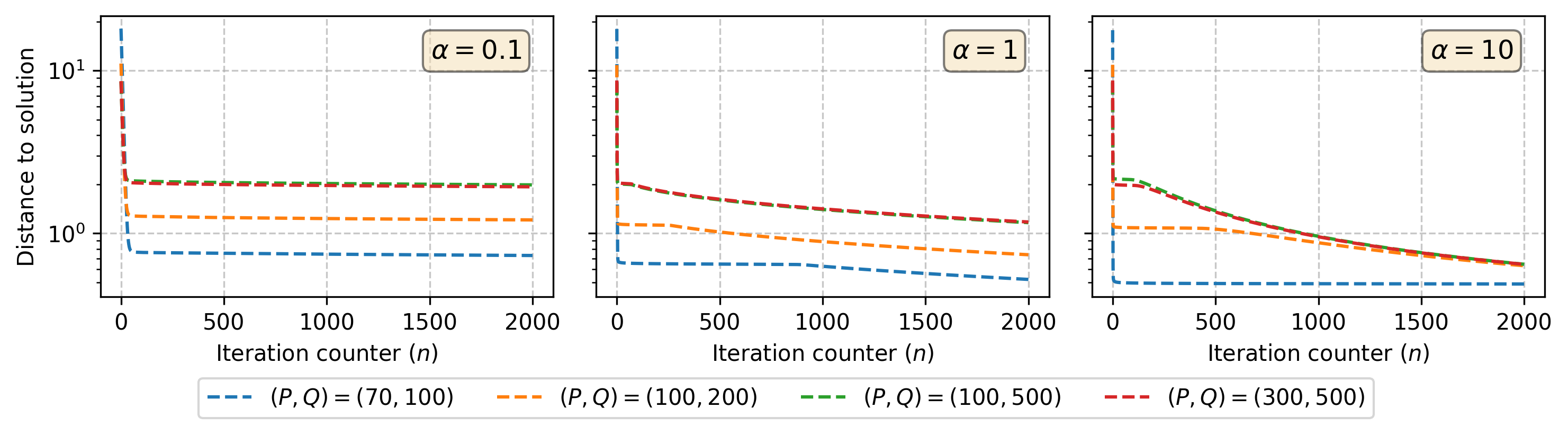}
    \caption{Evolution of error sequence $\|w_n-z\|$ for various problem dimensions and various proximal parameters.}
    \label{fig:LS_Dist}
\end{figure}

\begin{figure}[H]
    \centering
    \includegraphics[width=1\linewidth]{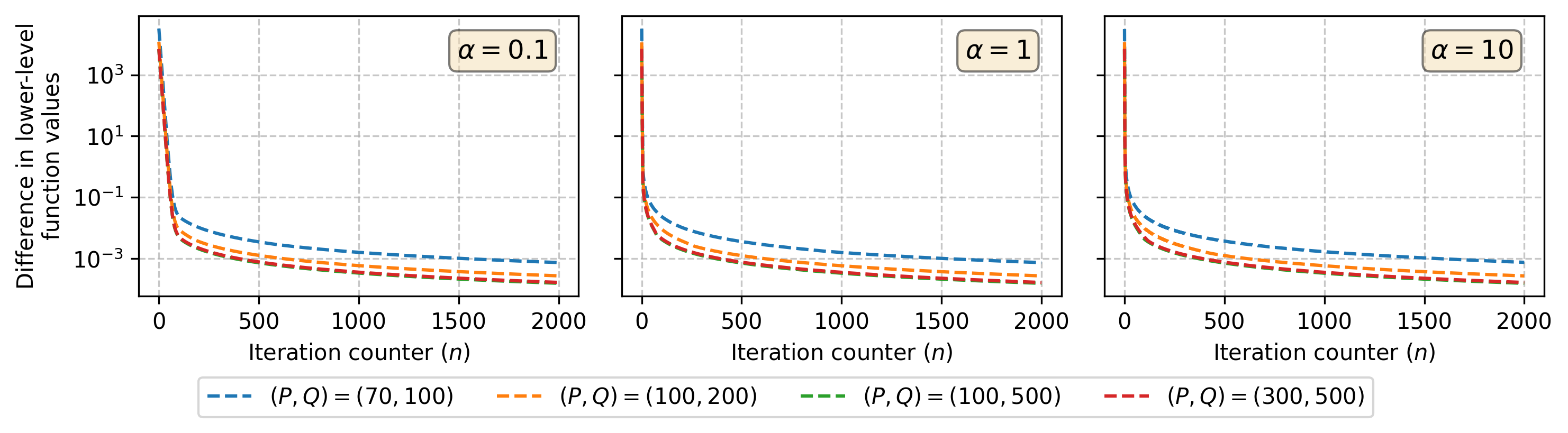}
    \caption{Evolution of lower-level function gap sequence $\tfrac12\|Aw_n-b\|^2-\tfrac12\|Az-b\|^2$ for various problem dimensions and various proximal parameters.}
    \label{fig:LS_Func}
\end{figure}

\subsection{Image Inpainting}\label{sec:image}

Finally, we consider the image inpainting problem, as done in \cite{maulen_inertial_2024,Cortild:2024aa}. We represent a grayscale image $Y$ of dimension $(P, Q)$ by a matrix in $X\eqdef [0, 1]^{P\times Q}\subset \mathcal H\eqdef \R^{P\times Q}$. We denote by $\Omega\in \{0, 1\}^{P\times Q}$ a mask such that $\Omega_{i, j}=0$ indicates that the pixel at position $(i, j)$ has been damaged. We denote by $\mathcal R$ the linear operator that maps an image to an image whose elements in $\Omega$ have been corrupted, namely $\mathcal R\colon \scrH\to \scrH, Y\mapsto \tilde Y$ with $\tilde Y_{i, j}=\Omega_{i, j}\cdot Y_{i, j}$. We note that $\mathcal R$ is a self-adjoint bounded linear operator with operator norm $1$. We define the corrupt image by $Y_\text{corrupt}\eqdef\mathcal R(Y)$. The aim of the image inpainting problem is to recover the original image $Y$ from the corrupt image $Y_\text{corrupt}$ and the corruption map $\mathcal R$. Mathematically, we consider the problem 
\[
\min_{Y\in X}\left\{\tfrac12\|\mathcal R(Y)-Y_\text{corrupt}\|^2+\sigma\|Y\|_*\right\},
\]
where $\|\cdot\|_*$ denotes the nuclear norm, which prones smoothness within the image, and $\sigma$ is a regularization parameter. Specifically, we seek the least-norm solution to the above problem, given by 
\[
\min_{Y\in \scrH}\left\{\tfrac12\|Y\|^2~\colon~ Y\in \argmin_{Y\in X}\left\{\tfrac12\|\mathcal R(Y)-Y_\text{corrupt}\|^2+\sigma\|Y\|_*\right\}\right\}.
\]
This matches Problem \eqref{eq:PTOS} with $\opG=\Id$, $\opA=\sigma \partial \|\cdot\|_*$, $\opB=\NC_X$, and $\opF(Y)=\mathcal R^{*}(\mathcal R(Y)-Y_\text{corrupt})$. We note that Assumption \ref{ass:TOS} is satisfied, and that the operator $\opT$ is nonexpansive per Lemma \ref{lem:TOS}. Although this setting does not strictly conform to our theoretical findings, its convergence properties are expected to be analogous due to the finite-dimensional nature of the problem, as detailed in Remark \ref{rem:TOS_nonexpansive}.

\paragraph{Experimental Setup.} For the inner loop, we set $\theta_k\equiv \theta=0.75$ and consider $\tau_k\equiv \tau$ to be constantly the largest value satisfying Equation \eqref{eq:parameterineq} across every outer iteration. We set $\beta_n=(n+1)^{-\eta}$ for $\eta=0.55$ and $\varepsilon_n=\bar\varepsilon\cdot (n+1)^{-2}$ for $\bar\varepsilon=2$. We run a total of $5000$ iterations, with an initial point $w_0=Y_\text{corrupt}$. We set the regularization parameter to be $\sigma=50$, and generate $\Omega$ randomly such that $20\%$ of the image is corrupted.

\paragraph{Results.} Figure \ref{fig:IP_Restored} shows the original image, the corrupt image, and the image restored using the above procedure for various proximal parameters. Table \ref{tab:IP} reports the total number of inner iterations run, along with the time spent to produce the results. Figure \ref{fig:IP_Curve} plots the lower-level objective across outer iterations for each value of $\alpha$, which is expected to converge to the lower-level objective value and not to $0$, as seen in the plot. We observe a loss in the solution quality for $\alpha=100$, although it is the fastest method. We also note that the solution qualities are very similar for $\alpha=0.1$ and $\alpha=1$, but the latter requires a larger number of total iterations.

\begin{figure}[H]
    \centering
    \includegraphics[width=1\linewidth]{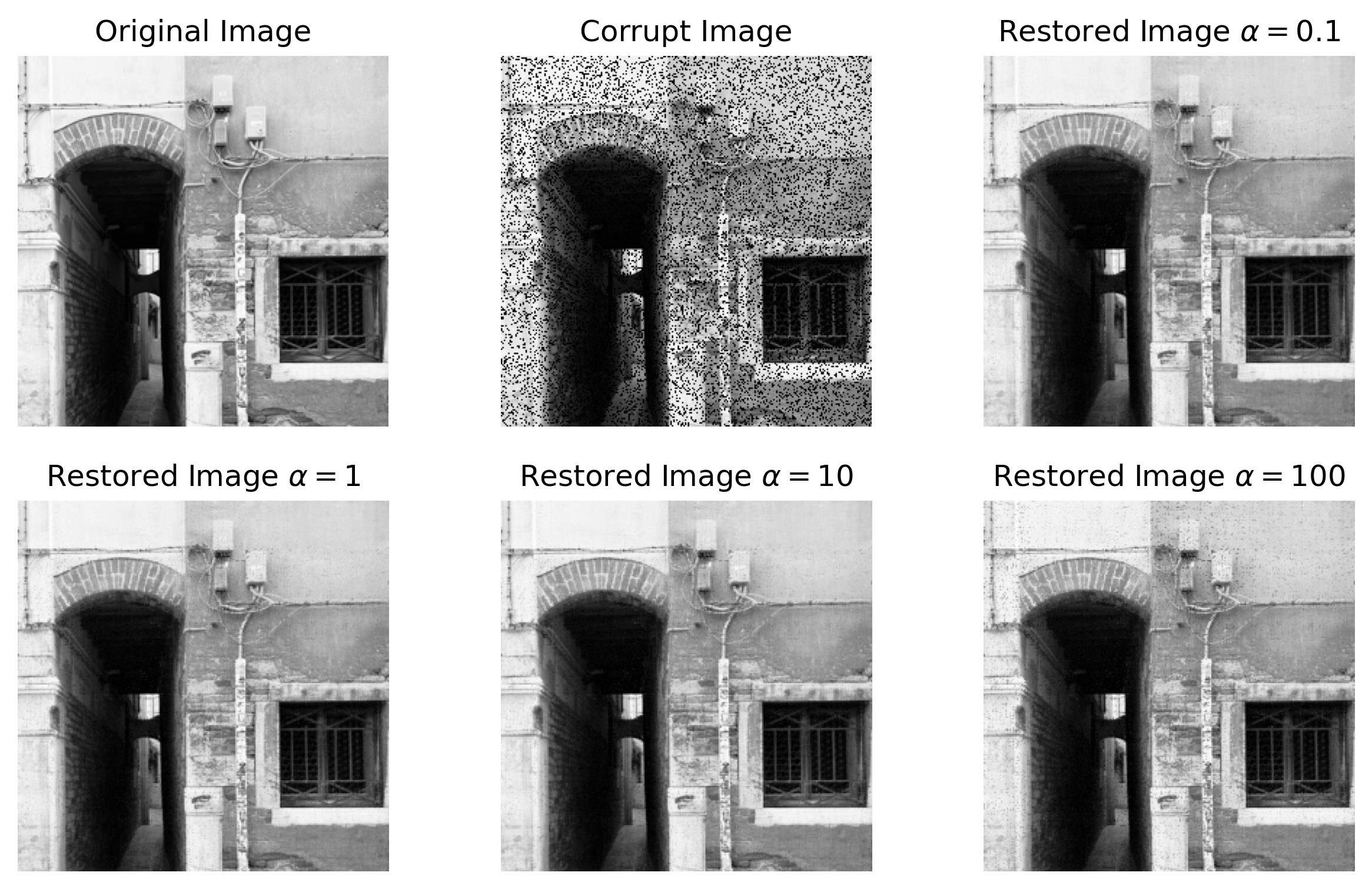}
    \caption{Result of image inpainting procedure for various proximal parameters.}
    \label{fig:IP_Restored}
\end{figure}

\begin{table}[H]
    \centering
    \caption{Number of inner iterations and clocktime to produce Figure \ref{fig:IP_Restored}.}
    \begin{tabular}{|c|c|c|c|c|}\hline
         & $\alpha=0.1$ & $\alpha=1$ & $\alpha=10$ & $\alpha=100$ \\ \hline
        Total inner iterations & 495096 & 156365 & 92867 & 73901 \\\hline
        Time in minutes & 294 & 80 & 38 & 36 \\\hline
    \end{tabular}
    \label{tab:IP}
\end{table}

\begin{figure}[H]
    \centering
    \includegraphics[width=.5\linewidth]{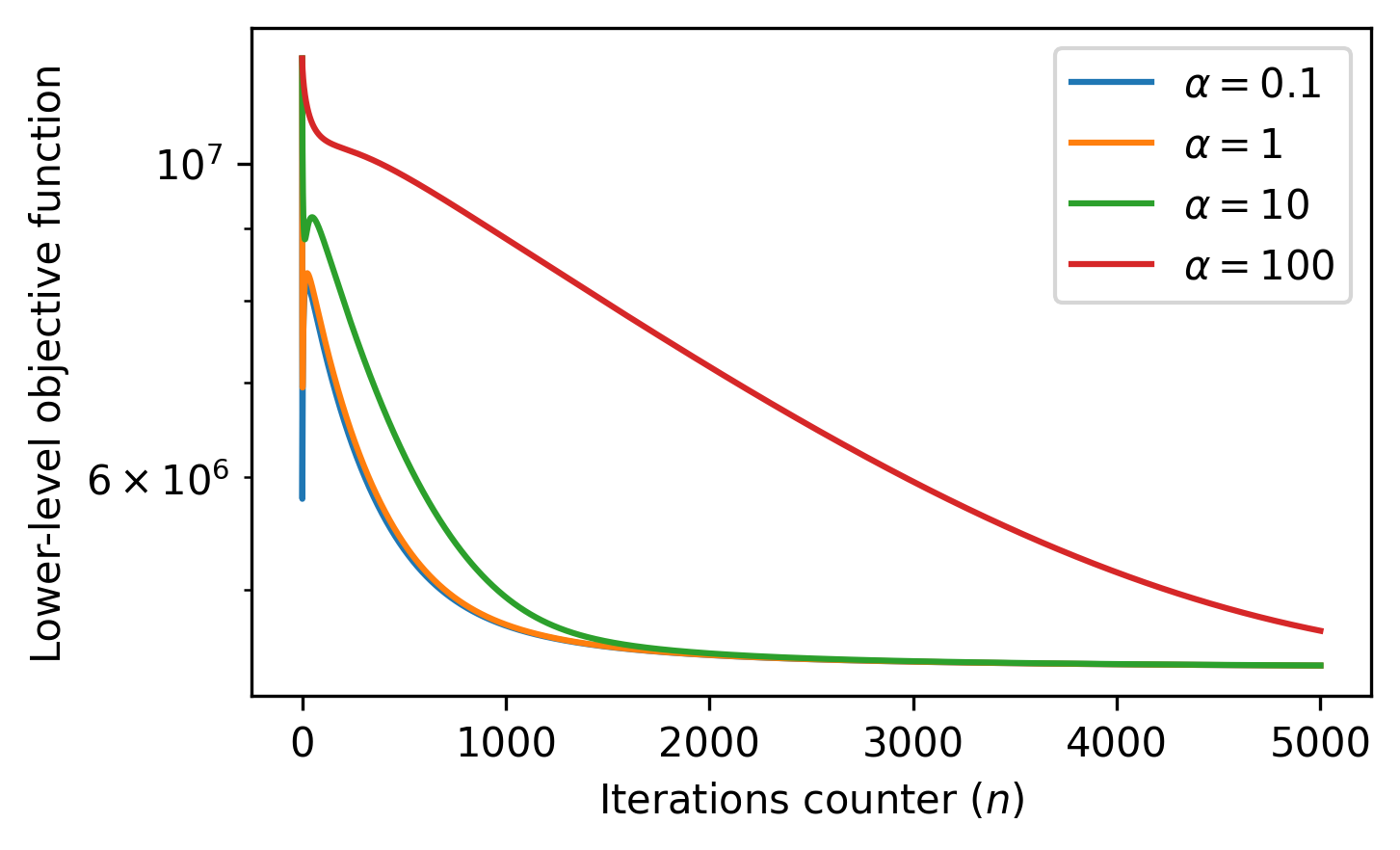}
    \caption{Evolution of lower-level objective function $\tfrac{1}{2}\|\mathcal R(w_n)-Y_{\text{corrupt}}\|^2+\sigma \| w_n\|_*$ for various proximal parameters.}
    \label{fig:IP_Curve}
\end{figure}


\section{Conclusion}\label{sec:conc}

In this work, we propose a double-loop path-following method to solve hierarchical variational inequalities. The inner loop aims to solve an auxiliary problem and, given a fixed-point encoding of said problem, employs a Krasnoselskii-Mann-type iteration. This step is very flexible regarding the chosen encoding strategy, such that our method includes various known algorithms. The outer loop takes the form of a restarting procedure, giving rise to a diagonal equilibrium tracking method (\texttt{DANTE}). We obtain weak convergence of the averaged iterates as well as convergence rates, which are comparable to the state-of-the-art results for \emph{specific} numerical schemes. 

We believe that our work opens up possibilities for various interesting directions for future research. First, it would be of practical interest to devise a fully adaptive scheme without restarts. Second, an important extension is the inclusion of stochastic data. These directions are the subject of ongoing research.

\paragraph{Acknowledgements}
Daniel Cortild acknowledges the support of the Clarendon Funds Scholarships. Meggie Marschner and Mathias Staudigl acknowledge financial support from the PGMO and the Deutsche Forschungsgemeinschaft (DFG) - Projektnummer 556222748 "non-stationary hierarchical minimization".

\bibliographystyle{abbrv}
\bibliography{references}

@article{Al-Homidan:2017aa,
  title        = {{Weak sharp solutions for generalized variational inequalities}},
  author       = {Al-Homidan, Suliman and Ansari, Qamrul Hasan and Burachik, Regina S.},
  year         = 2017,
  journal      = {Positivity},
  volume       = 21,
  number       = 3,
  pages        = {1067--1088},
  doi          = {10.1007/s11117-016-0453-x},
  url          = {https://doi.org/10.1007/s11117-016-0453-x}
}

@article{liu2021investigating,
  title        = {{Investigating bi-level optimization for learning and vision from a unified perspective: A survey and beyond}},
  author       = {Liu, Risheng and Gao, Jiaxin and Zhang, Jin and Meng, Deyu and Lin, Zhouchen},
  year         = 2021,
  journal      = {IEEE Transactions on Pattern Analysis and Machine Intelligence},
  publisher    = {IEEE},
  volume       = 44,
  number       = 12,
  pages        = {10045--10067}
}

@misc{boct2025accelerating,
  title        = {{Accelerating Diagonal Methods for Bilevel Optimization: Unified Convergence via Continuous-Time Dynamics}},
  author       = {Bo{\c{t}}, Radu Ioan and Chenchene, Enis and Csetnek, Ern{\"o} Robert and Hulett, David Alexander},
  year         = 2025,
  note         = {arXiv preprint arXiv:2505.14389}
}

@article{alart1991penalization,
  title        = {{Penalization in non-classical convex programming via variational convergence}},
  author       = {Alart, P and Lemaire, Bernard},
  year         = 1991,
  journal      = {Mathematical Programming},
  publisher    = {Springer},
  volume       = 51,
  number       = 1,
  pages        = {307--331}
}

@article{MerchavSIOPT23,
  title        = {{Convex Bi-level Optimization Problems with Nonsmooth Outer Objective Function}},
  author       = {Merchav, Roey and Sabach, Shoham},
  year         = 2023,
  journal      = {SIAM Journal on Optimization},
  volume       = 33,
  number       = 4,
  pages        = {3114--3142},
  doi          = {10.1137/22M1533608}
}

@article{drusvyatskiy2018error,
  title        = {{Error bounds, quadratic growth, and linear convergence of proximal methods}},
  author       = {Drusvyatskiy, Dmitriy and Lewis, Adrian S},
  year         = 2018,
  journal      = {Mathematics of operations research},
  publisher    = {INFORMS},
  volume       = 43,
  number       = 3,
  pages        = {919--948}
}

@article{staudigl2023random,
  title        = {{Random block-coordinate methods for inconsistent convex optimisation problems}},
  author       = {Staudigl, Mathias and Jacquot, Paulin},
  year         = 2023,
  journal      = {Fixed Point Theory and Algorithms for Sciences and Engineering},
  volume       = 2023,
  number       = 14
}

@article{beck2014first,
  title        = {{A first order method for finding minimal norm-like solutions of convex optimization problems}},
  author       = {Beck, Amir and Sabach, Shoham},
  year         = 2014,
  journal      = {Mathematical Programming},
  publisher    = {Springer},
  volume       = 147,
  number       = 1,
  pages        = {25--46}
}

@article{shen2023online,
  title        = {{An online convex optimization-based framework for convex bilevel optimization}},
  author       = {Shen, Lingqing and Ho-Nguyen, Nam and K{\i}l{\i}n{\c{c}}-Karzan, Fatma},
  year         = 2023,
  journal      = {Mathematical Programming},
  publisher    = {Springer},
  volume       = 198,
  number       = 2,
  pages        = {1519--1582}
}

@article{marino2006general,
  title        = {{A general iterative method for nonexpansive mappings in Hilbert spaces}},
  author       = {Marino, Giuseppe and Xu, Hong-Kun},
  year         = 2006,
  journal      = {Journal of Mathematical Analysis and Applications},
  publisher    = {Elsevier},
  volume       = 318,
  number       = 1,
  pages        = {43--52}
}

@article{deutsch1998minimizing,
  title        = {{Minimizing certain convex functions over the intersection of the fixed point sets of nonexpansive mappings}},
  author       = {Deutsch, Frank and Yamada, Isao},
  year         = 1998,
  journal      = {Numerical Functional Analysis and Optimization},
  publisher    = {[New York] M. Dekker.},
  volume       = 19,
  number       = 1,
  pages        = {33--56}
}

@inproceedings{yousefian2021bilevel,
  title        = {{Bilevel distributed optimization in directed networks}},
  author       = {Yousefian, Farzad},
  year         = 2021,
  booktitle    = {2021 American Control Conference},
  pages        = {2230--2235}
}

@book{dempe2020bilevel,
  title        = {{Bilevel optimization}},
  author       = {Dempe, Stephan and Zemkoho, Alain},
  year         = 2020,
  booktitle    = {Springer optimization and its applications},
  publisher    = {Springer},
  volume       = 161
}

@article{pang1997error,
  title        = {{Error bounds in mathematical programming}},
  author       = {Pang, Jong-Shi},
  year         = 1997,
  journal      = {Mathematical Programming},
  volume       = 79,
  number       = 1,
  pages        = {299--332}
}

@misc{alves_inertial_2025,
  title        = {{An Inertial Iteratively Regularized Extragradient Method for Bilevel Variational Inequality Problems}},
  author       = {Alves, M. Marques and Chen, Kangming and Fukuda, Ellen H.},
  year         = 2025,
  publisher    = {arXiv},
  number       = {arXiv:2507.16640},
  doi          = {10.48550/arXiv.2507.16640},
  note         = {arXiv preprint arXiv:2507.16640}
}

@article{burke_weak_1993,
  title        = {{Weak {{Sharp Minima}} in {{Mathematical Programming}}}},
  author       = {Burke, J. V. and Ferris, M. C.},
  year         = 1993,
  journal      = {SIAM Journal on Control and Optimization},
  volume       = 31,
  number       = 5,
  pages        = {1340--1359},
  doi          = {10.1137/0331063}
}

@article{marcotte_weak_1998,
  title        = {{Weak {{Sharp Solutions}} of {{Variational Inequalities}}}},
  author       = {Marcotte, Patrice and Zhu, Daoli},
  year         = 1998,
  journal      = {SIAM Journal on Optimization},
  publisher    = {{Society for Industrial and Applied Mathematics}},
  volume       = 9,
  number       = 1,
  pages        = {179--189},
  doi          = {10.1137/S1052623496309867}
}

@article{moursi_douglas_2019,
  title        = {{Douglas--{{Rachford Splitting}} for the {{Sum}} of a {{Lipschitz Continuous}} and a {{Strongly Monotone Operator}}}},
  author       = {Moursi, Walaa M. and Vandenberghe, Lieven},
  year         = 2019,
  journal      = {Journal of Optimization Theory and Applications},
  volume       = 183,
  number       = 1,
  pages        = {179--198},
  doi          = {10.1007/s10957-019-01517-8}
}

@misc{marschner_tikhonov_2025,
  title        = {{Tikhonov Regularized Exterior Penalty Methods for Hierarchical Variational Inequalities}},
  author       = {Marschner, Meggie and Staudigl, Mathias},
  year         = 2025,
  note         = {arXiv preprint arXiv:2508.20872}
}

@article{attouch_backward_2018,
  title        = {{Backward--Forward Algorithms for Structured Monotone Inclusions in {{Hilbert}} Spaces}},
  author       = {Attouch, H{\'e}dy and Peypouquet, Juan and Redont, Patrick},
  year         = 2018,
  journal      = {Journal of Mathematical Analysis and Applications},
  volume       = 457,
  number       = 2,
  pages        = {1095--1117},
  doi          = {10.1016/j.jmaa.2016.06.025}
}

@article{lions_splitting_1979,
  title        = {{Splitting {{Algorithms}} for the {{Sum}} of {{Two Nonlinear Operators}}}},
  author       = {Lions, P. L. and Mercier, B.},
  year         = 1979,
  journal      = {SIAM Journal on Numerical Analysis},
  publisher    = {{Society for Industrial and Applied Mathematics}},
  volume       = 16,
  number       = 6,
  pages        = {964--979},
  doi          = {10.1137/0716071}
}

@article{passty_ergodic_1979,
  title        = {{Ergodic Convergence to a Zero of the Sum of Monotone Operators in {{Hilbert}} Space}},
  author       = {Passty, Gregory B},
  year         = 1979,
  journal      = {Journal of Mathematical Analysis and Applications},
  volume       = 72,
  number       = 2,
  pages        = {383--390},
  doi          = {10.1016/0022-247X(79)90234-8}
}

@article{maulen_inertial_2024,
  title        = {{Inertial {{Krasnoselskii-Mann Iterations}}}},
  author       = {Maul{\'e}n, Juan Jos{\'e} and Fierro, Ignacio and Peypouquet, Juan},
  year         = 2024,
  journal      = {Set-Valued and Variational Analysis},
  publisher    = {Springer},
  volume       = 32,
  number       = 10,
  doi          = {10.1007/s11228-024-00713-7}
}

@article{davis_threeoperator_2017a,
  title        = {{A {{Three-Operator Splitting Scheme}} and Its {{Optimization Applications}}}},
  author       = {Davis, Damek and Yin, Wotao},
  year         = 2017,
  journal      = {Set-Valued and Variational Analysis},
  volume       = 25,
  number       = 4,
  pages        = {829--858},
  doi          = {10.1007/s11228-017-0421-z}
}

@article{hintermuller2011first,
  title        = {{First-order optimality conditions for elliptic mathematical programs with equilibrium constraints via variational analysis}},
  author       = {Hinterm{\"u}ller, Michael and Surowiec, Thomas},
  year         = 2011,
  journal      = {SIAM Journal on Optimization},
  publisher    = {SIAM},
  volume       = 21,
  number       = 4,
  pages        = {1561--1593}
}

@article{de2023bilevel,
  title        = {{Bilevel imaging learning problems as mathematical programs with complementarity constraints: Reformulation and theory}},
  author       = {De los Reyes, Juan Carlos},
  year         = 2023,
  journal      = {SIAM Journal on Imaging Sciences},
  publisher    = {SIAM},
  volume       = 16,
  number       = 3,
  pages        = {1655--1686}
}

@article{outrata2000generalized,
  title        = {{A generalized mathematical program with equilibrium constraints}},
  author       = {Outrata, Jiri V},
  year         = 2000,
  journal      = {SIAM Journal on Control and Optimization},
  publisher    = {SIAM},
  volume       = 38,
  number       = 5,
  pages        = {1623--1638}
}

@article{hintermuller2014several,
  title        = {{Several approaches for the derivation of stationarity conditions for elliptic MPECs with upper-level control constraints}},
  author       = {Hinterm{\"u}ller, Michael and Mordukhovich, Boris S and Surowiec, Thomas M},
  year         = 2014,
  journal      = {Mathematical Programming},
  publisher    = {Springer},
  volume       = 146,
  number       = 1,
  pages        = {555--582}
}

@article{hintermuller2015bilevel,
  title        = {{Bilevel optimization for calibrating point spread functions in blind deconvolution.}},
  author       = {Hinterm{\"u}ller, Michael and Wu, Tao},
  year         = 2015,
  journal      = {Inverse Problems \& Imaging},
  volume       = 9,
  number       = 4
}

@book{Luo_Pang_Ralph_1996,
  title        = {{Mathematical Programs with Equilibrium Constraints}},
  author       = {Luo, Zhi-Quan and Pang, Jong-Shi and Ralph, Daniel},
  year         = 1996,
  publisher    = {Cambridge University Press},
  place        = {Cambridge}
}

@article{samadi_improved_2025,
  title        = {{Improved {{Guarantees}} for {{Optimal Nash Equilibrium Seeking}} and {{Bilevel Variational Inequalities}}}},
  author       = {Samadi, Sepideh and Yousefian, Farzad},
  year         = 2025,
  journal      = {SIAM Journal on Optimization},
  publisher    = {{Society for Industrial and Applied Mathematics}},
  volume       = 35,
  number       = 1,
  pages        = {369--399},
  doi          = {10.1137/23M1589402}
}

@article{Cabot2005,
  title        = {{Proximal Point Algorithm Controlled by a Slowly Vanishing Term: Applications to Hierarchical Minimization}},
  author       = {Cabot, Alexandre},
  year         = 2005,
  journal      = {SIAM Journal on Optimization},
  volume       = 15,
  number       = 2,
  pages        = {555--572},
  doi          = {10.1137/S105262340343467X},
  url          = {https://doi.org/10.1137/S105262340343467X}
}

@article{SabSht17,
  title        = {{A First Order Method for Solving Convex Bilevel Optimization Problems}},
  author       = {Sabach, Shoham and Shtern, Shimrit},
  year         = 2017,
  journal      = {SIAM Journal on Optimization},
  volume       = 27,
  number       = 2,
  pages        = {640--660},
  doi          = {10.1137/16M105592X}
}

@article{Benenati:2022aa,
  title        = {{On the optimal selection of generalized Nash equilibria in linearly coupled aggregative games}},
  author       = {Benenati, Emilio and Ananduta, Wicak and Grammatico, Sergio},
  year         = 2022,
  journal      = {IEEE 61st Conference on Decision and Control},
  publisher    = {IEEE},
  address      = {Cancun, Mexico},
  pages        = {6389--6394},
  doi          = {10.1109/CDC51059.2022.9993415}
}

@article{Benenati:2023aa,
  title        = {{A semi-decentralized tikhonov-based algorithm for optimal generalized nash equilibrium selection}},
  author       = {Benenati, Emilio and Ananduta, Wicak and Grammatico, Sergio},
  year         = 2023,
  journal      = {IEEE 62nd Conference on Decision and Control},
  publisher    = {IEEE},
  pages        = {4243--4248}
}

@article{Kaushik:2021aa,
  title        = {{A method with convergence rates for optimization problems with variational inequality constraints}},
  author       = {Kaushik, Harshal D and Yousefian, Farzad},
  year         = 2021,
  journal      = {SIAM Journal on Optimization},
  publisher    = {SIAM},
  volume       = 31,
  number       = 3,
  pages        = {2171--2198}
}

@article{Zhao:2009aa,
  title        = {{The composite absolute penalties family for grouped and hierarchical variable selection}},
  author       = {Peng Zhao and Guilherme Rocha and Bin Yu},
  year         = 2009,
  journal      = {The Annals of Statistics},
  volume       = 37,
  number       = {6A},
  pages        = {3468--3497},
  doi          = {10.1214/07-AOS584}
}

@article{Chambolle:2004aa,
  title        = {{An Algorithm for Total Variation Minimization and Applications}},
  author       = {Chambolle, Antonin},
  year         = 2004,
  journal      = {Journal of Mathematical Imaging and Vision},
  volume       = 20,
  number       = 1,
  pages        = {89--97},
  doi          = {10.1023/B:JMIV.0000011325.36760.1e}
}

@article{Cortild:2024aa,
  title        = {{Krasnoselskii--{{Mann Iterations}}: {{Inertia}}, {{Perturbations}} and {{Approximation}}}},
  author       = {Cortild, Daniel and Peypouquet, Juan},
  year         = 2025,
  journal      = {Journal of Optimization Theory and Applications},
  volume       = 204,
  number       = 35,
  doi          = {10.1007/s10957-024-02600-5}
}

@article{Alvarez:2001aa,
  title        = {{An Inertial Proximal Method for Maximal Monotone Operators via Discretization of a Nonlinear Oscillator with Damping}},
  author       = {Alvarez, Felipe and Attouch, Hedy},
  year         = 2001,
  journal      = {Set-Valued Analysis},
  volume       = 9,
  number       = 1,
  pages        = {3--11},
  doi          = {10.1023/A:1011253113155}
}

@article{Facchinei:2014aa,
  title        = {{VI-constrained hemivariational inequalities: distributed algorithms and power control in ad-hoc networks}},
  author       = {Facchinei, Francisco and Pang, Jong-Shi and Scutari, Gesualdo and Lampariello, Lorenzo},
  year         = 2014,
  journal      = {Mathematical Programming},
  volume       = 145,
  number       = 1,
  pages        = {59--96},
  doi          = {10.1007/s10107-013-0640-5}
}

@article{Dempe:2021aa,
  title        = {{Simple bilevel programming and extensions}},
  author       = {Dempe, Stephan and Dinh, Nguyen and Dutta, Joydeep and Pandit, Tanushree},
  year         = 2021,
  journal      = {Mathematical Programming},
  volume       = 188,
  number       = 1,
  pages        = {227--253},
  doi          = {10.1007/s10107-020-01509-x}
}

@article{Lampariello:2020aa,
  title        = {{An explicit Tikhonov algorithm for nested variational inequalities}},
  author       = {Lampariello, Lorenzo and Neumann, Christoph and Ricci, Jacopo M. and Sagratella, Simone and Stein, Oliver},
  year         = 2020,
  journal      = {Computational Optimization and Applications},
  volume       = 77,
  number       = 2,
  pages        = {335--350},
  doi          = {10.1007/s10589-020-00210-1}
}

@inproceedings{Yamada2011MinimizingTM,
  title        = {{Minimizing the Moreau Envelope of Nonsmooth Convex Functions over the Fixed Point Set of Certain Quasi-Nonexpansive Mappings}},
  author       = {Isao Yamada and Masahiro Yukawa and Masao Yamagishi},
  year         = 2011,
  booktitle    = {Fixed-Point Algorithms for Inverse Problems in Science and Engineering},
  doi          = {10.1007/978-1-4419-9569-8_17}
}

@inproceedings{AminiYous19,
  title        = {{An Iterative Regularized Incremental Projected Subgradient Method for a Class of Bilevel Optimization Problems}},
  author       = {Amini, Mostafa and Yousefian, Farzad},
  year         = 2019,
  booktitle    = {2019 American Control Conference},
  pages        = {4069--4074},
  doi          = {10.23919/ACC.2019.8814637}
}

@book{BauCom16,
  title        = {{Convex Analysis and Monotone Operator Theory in Hilbert Spaces}},
  author       = {Heinz H. Bauschke and Patrick L. Combettes},
  year         = 2016,
  publisher    = {Springer - CMS Books in Mathematics},
  doi          = {10.1007/978-3-319-48311-5}
}

\end{document}